\documentclass{amsart}
\usepackage{amsthm,amssymb,mathtools}
\usepackage{stmaryrd,tikz,pgfplots}
\usepackage{marginnote}

\usepackage[textsize=tiny]{todonotes}

\newcommand{\Marginpar}[1]{\marginpar{\tiny{#1}}}
\newcommand{\Note}[1]{{\par\noindent\hrulefill\par\tiny{#1}\par\noindent\hrulefill\par}}
\newcommand{\Detail}[1]{{#1}}
\renewcommand{\Marginpar}[1]{}
\renewcommand{\Note}[1]{}
\renewcommand{\Detail}[1]{}

\usepackage{hyperref}
\usepackage[all]{xy}
\usepackage{amsmath,amsfonts}	
\usepackage{amsthm,latexsym,amssymb}
\usepackage{tensor}

\newtheorem{thm}{Theorem}[section]
\newtheorem{prop}[thm]{Proposition}
\newtheorem{lem}[thm]{Lemma}
\newtheorem{cor}[thm]{Corollary}

\newtheorem{CYproblem}[thm]{Chern--Yamabe Problem}
\newtheorem{EXproblem}[thm]{Existence Problem}

\theoremstyle{definition}
\newtheorem{defn}[thm]{Definition}

\newtheorem{rem}[thm]{Remark}

\renewcommand{\[}{\begin{equation*}}
\renewcommand{\]}{\end{equation*}}

\DeclareMathOperator{\grad}{grad}

\def\C{\mathbb{C}}

\def\LieHam{{\mathfrak{ham}^f(\omega)}}

\DeclareMathOperator\Lie{Lie}
\DeclareMathOperator\Ham{Ham}
\DeclareMathOperator\tr{tr}

\DeclareMathOperator\End{End}

\DeclareMathOperator\vol{vol}

\begin{document}
\parskip1mm

\title[Conformally Almost Hermitian Geometry]{Integrability Theorems and Conformally Constant Chern Scalar curvature Metrics in Almost Hermitian Geometry}

\author{Mehdi Lejmi}
\address{Department of Mathematics, Bronx Community College of CUNY, Bronx, NY 10453, USA.}
\email{mehdi.lejmi@bcc.cuny.edu}
\author{Markus Upmeier}
\address{Universit\"atsstrasse 14, 86159 Augsburg, Germany.}
\email{Markus.Upmeier@math.uni-augsburg.de}

\begin{abstract}
The various scalar curvatures on an almost Hermitian manifold are studied, in particular with respect to conformal variations. We show several integrability theorems, which state that two of these can only agree in the K\"ahler case. Our main question is the existence of almost K\"ahler metrics with conformally constant Chern scalar curvature. This problem is completely solved for ruled manifolds and in a complementary case where methods from the Chern--Yamabe problem are adapted to the non-integrable case. Also a moment map interpretation of the problem is given, leading to a Futaki invariant and the usual picture from geometric invariant theory.
\end{abstract}
\maketitle

\setcounter{tocdepth}{1}
\tableofcontents


The present paper is devoted to the conformal geometry of almost Hermitian structures, in particular to aspects relating to their scalar curvature.

The necessary background is briefly reviewed in \S\ref{sec:Preliminaries}. In particular, we recall the Chern connection and its torsion (see~\cite{MR1456265}) on almost Hermitian manifolds, which reflects also the almost complex structure. From it one derives three Ricci forms and two scalar curvatures: the Hermitian (or Chern) scalar curvature $s^H=2s^C$ and the third scalar curvature $s$. From the Levi-Civita connection we also have the Riemannian scalar curvature $s^g$ and all three evidently coincide in the K\"ahler case. In \S\ref{sec:ComparisionScalar} their precise relationship in general is established by careful calculation in local coordinates (see Propositions~\ref{prop:ScalarComparision}, \ref{prop:ScalarComparision1}, \ref{sHsG}). The formulas generalize those of Gauduchon in the integrable case \cite{MR742896}. 

These are applied in \S\ref{sec:VanishingTheorems} to prove several new integrability theorems, which assert that when two scalar curvatures coincide we must already be in the K\"ahler case. These holds in any dimension when we have a nearly K\"ahler structure (Corollary~\ref{nearlyKaehlerVanish}). We also have results in any dimension in the almost K\"ahler case (see Corollary~\ref{cor:VanishAK} and also Apostolov--Dr\u{a}ghici~\cite{apostolov1999almost}). The completely general almost Hermitian case is restricted to dimension $4$ (Theorem~\ref{THM:vanishing}). We also obtain an interesting result (Corollary~\ref{curious-cor}) on $6$\nobreakdash-dimensional compact non-K\"ahler, nearly K\"ahler manifolds: they all have $s^H=0$.

We then compute in \S\ref{sec:ConformalChanges} the behaviour of our Ricci forms and scalar curvatures under conformal variations (see Corollaries~\ref{RicciVar} and~\ref{cor:ConformalScalVar}). This allows us to prove another integrability theorem (Theorem~\ref{ineq_chern_riem_con_AK}) for conformally almost K\"ahler structures, relating the Hermitian and Riemannian scalar curvature. 

In \S\ref{sec:conformally} we state the basic problem that will concern us for the rest of the paper: an almost Hermitian structure is \emph{conformally constant} if some conformal variation has constant Hermitian scalar curvature. We first extend the results of  Angella--Calamai--Spotti \cite{Ang-Cal-Spo} on the Chern--Yamabe problem to the non-integrable case, and show some independent results of interest. In Corollary~\ref{corExtend} we obtain that every almost Hermitian structure with non-positive fundamental constant \eqref{GauduchonDegree} is conformally constant. The remaining case is much more difficult. It is not even known in general whether {\it one can find any conformally constant almost Hermitian structure}. This is our Existence Problem~\ref{EXproblem}, where we restrict to the symplectic case.

In \S\ref{sec:BabyCase} we solve this problem for ruled manifolds given by the generalized Calabi construction (Theorem~\ref{thm:babycase}). Drawing on the fundamental work by Apostolov--Calderbank--Gauduchon--T{\o}nnesen-Friedman \cite{MR2144249,MR2807093}, the proof is carried out in \S\ref{ssec:ExistenceProblem} and quickly reduces to an ODE for a metric on the moment polytope, an interval in our case. The main difficulty is to show positivity of the solution and this is done by a careful asymptotic analysis in \S\ref{ssec:analyticpart}. The manifolds thus constructed are new examples of non-K\"ahler structures of constant Hermitian scalar curvature with positive fundamental constant

Finally in \S\ref{sec:MomentMap} we give an interpretation of our existence problem in the framework of moment maps (Theorem~\ref{thm:momentmap}). Here we assume a symmetry on the manifold, namely a Hamiltonian vector field which is Killing for some metric.  

This leads to the usual existence and uniqueness conjectures in terms of geometric invariant theory, 
and also to a Futaki invariant (see~\S\ref{ssec:Futaki}). We end in \S\ref{ssec:toriccase} with concrete calculations in the toric case.

%

\subsection*{Acknowledgements}
Both authors are grateful to Tedi Dr\u{a}ghici and Daniele Angella for their suggestions and interest in this work.

\section{Preliminaries}\label{sec:Preliminaries}

Let $(M,J,g)$ be an almost Hermitian manifold of real dimension $m=2n$. Thus ${J\colon TM\to TM}$ is an almost complex structure $J^2=-1$ that is orthogonal for the Riemannian metric $g$. The associated fundamental form is $F=g(J\cdot,\cdot)$. We usually do not distinguish between the metric and the almost complex structure and write $g_J\coloneqq F(\cdot,J\cdot)$ for the metric corresponding to $J$. The volume form is $\vol_g=\frac{F^n}{n!}$. On the complexification $TM\otimes \mathbb{C} = T^{1,0} \oplus T^{0,1}$ we consider the $\mathbb{C}$-bilinear extension of $g$, the Hermitian form $h(X\otimes z, Y\otimes w) = z\overline{w} g(X,Y)$, and the restriction of $h$ to $TM$, which we identify with $T^{1,0}$.

\subsection{Complex Coordinates}\label{ComplexCoordinates}
Let $z_\alpha$ denote a complex basis of $T^{1,0}$. Then $\bar{z}_{\bar\alpha}$ is the basis of $T^{0,1}$ obtained by conjugation. The dual basis is denoted $z^\alpha, \bar{z}^{\bar\alpha}$. The components of the Hermitian form are\Marginpar{Bad, but standard notation. $h_{\alpha\bar\beta}=g_\mathbb{C}(z_\alpha,\bar{z}_{\bar\beta})$ should be called $g_{\alpha\bar\beta}$. The bar over $\beta$ has no meaning, except to signify that $h$ is anti-linear in the second component. It is introduced to match the summation convention.}
\[
	h_{\alpha\bar\beta} = h(z_\alpha, z_\beta).
\]
The transposed inverse of $h_{\alpha\bar\beta}$ is denoted $h^{\alpha\bar\beta}$. Thus $h_{\alpha\bar\gamma}h^{\beta\bar\gamma}=\delta_\alpha^\beta$ and $h_{\alpha\bar\beta} = \overline{h_{\beta\bar\alpha}}$. The fundamental form is then $F=ih_{\alpha\bar\beta} z^\alpha\wedge \bar{z}^{\bar\beta}$. We shall use the Hermitian form to raise and lower tensor indices.

We also get a $J$-adapted orthonormal frame ${e_1, e_2=Je_1, \ldots, e_{m-1}, e_m=Je_{m-1}}$ of the real tangent bundle $TM$ by decomposing $z_\alpha$ into the real and imaginary part:
\begin{equation}\label{Jadapted}
	z_\alpha = \frac12(e_{2\alpha-1} - ie_{2\alpha})
\end{equation}
By convention $\alpha, \beta, \gamma, \ldots$ range over $1,\ldots, n$ while $i,j,k,\ldots$ range over $1,\ldots, m$.

The twisted exterior differential of a $p$-form $\psi$ is defined as (see~\cite[(1.11.1)]{gauduchon2010calabi})
\[
	d^c\psi = J\psi J^{-1}\psi,
\]
where $J$ acts on forms by $(J^{-1})^*$ (some authors have a different sign convention).

\subsection{Type Decomposition}

Let $E$ be a vector bundle on $M$ with complex structure $J^E$. Unless $E=\C$ the space $\Omega^p(M;E)$ of $E$-valued differential $p$-forms has two different type decompositions.

\begin{defn}
A form $\psi \in \Omega^p(M;E)$ has \emph{$E$-type $(r,s)$} when $p=r+s$ and
\begin{align}\label{Type1}
	\sum_{k=1}^p \psi_{X_1 \ldots JX_k\ldots X_p} = (r-s) J^E\left(\psi_{X_1\ldots X_p} \right)
	\qquad \forall X_i \in TM.
\end{align}
The subspace of forms of $E$-type $(r,s)$ is denoted by $\Omega^{r,s}(M;E)$. We write $\psi^{r,s}$ for the projection with respect to this direct sum decomposition of $\Omega^p(M;E)$.
\end{defn}
Hence $\psi$ behaves like an ordinary $(r,s)$-form, except that it is vector-valued.
For example, the Nijenhuis tensor $N \in \Omega^2(M;TM)$ has $TM$-type $(0,2)$.

\begin{lem}
For a connection with $\nabla_X J=0$, $\nabla_X\psi$ has the same $E$-type as $\psi$.
\end{lem}

To understand the $E$-type with respect to contractions, let us say that $\psi$ has \emph{ordinary type $(r,s)$} when in a local frame as in Section~\ref{ComplexCoordinates} we may write\Marginpar{\tiny Either sum only over ascending indices $\alpha_1<\cdots<\alpha_r$ and $\beta_1 < \cdots < \beta_s$ and then ${\psi_{\alpha_1\ldots\alpha_r,\bar\beta_1\ldots\bar\beta_s}=\psi(z_{\alpha_1},\ldots,z_{\alpha_{r}},\bar{z}_{\bar\beta_1},\ldots,\bar{z}_{\bar\beta_s})}$. Or divide by $r!s!$ for this to remain true when summing over all indices}
\begin{align}\label{ordinarytype}
\psi = \frac{1}{r!s!}\psi_{\alpha_1\ldots\alpha_r,\bar\beta_1\ldots\bar\beta_s} z^{\alpha_1}\cdots z^{\alpha_r}\wedge \bar{z}^{\bar{\beta}_1} \cdots \bar{z}^{\bar{\beta}_s},
\end{align}
the coefficients being sections of $E$, anti-symmetric for $\alpha_1\ldots\alpha_r$ and for $\bar\beta_1\ldots\bar\beta_s$.
Thus, the ordinary type behaves as expected under contraction with $(1,0)$ and $(0,1)$-vector fields.
Concerning the $E$-type, we have the following observation:

\begin{lem}\label{lem:type}
A form $\psi$ has $E$-type $(r,s)$ precisely when it is the sum of an $E^{1,0}$-valued form of ordinary type $(r,s)$ and an $E^{0,1}$-valued form of ordinary type $(s,r)$.
\end{lem}

Finally, in case $E=TM$ we may use the metric $g$ to identify $(p+1)$-forms with $TM$-valued $p$-forms. Using the musical isomorphism we get a map
\begin{equation}\label{pForm_TM_Form}
i\colon \Omega^{p+1}(M)\hookrightarrow \Omega^{p}(M;TM),\quad
	i(\phi)_{X_1\ldots X_p} = {\phi_{-, X_1\ldots X_p}}^{\sharp_g}.
\end{equation}
Note that $\phi\in \Omega^{p+1}(M)$ are \emph{real} forms. From \eqref{pForm_TM_Form} we get a third type decomposition. This has been used by Gauduchon \cite[(1.3.2)]{MR1456265} in the case $p=2$.

\begin{defn}
A $(r+s)$-form $\phi$ has \emph{real type} $(r,s)+(s,r)$ when the complexification of $\phi$ is a sum of a complex $(r,s)$-form and (its conjugate) $(s,r)$-form. \Marginpar{Only the $(r,s)+(s,r)$ grading descends to real forms. This is because a real form can never be of pure type.} We write $\Omega^{(r,s)+(s,r)}(M)$ for the space of real forms of real type $(r,s)+(s,r)$.
\end{defn}

\begin{lem}\label{lem:realTMtype}
The map \eqref{pForm_TM_Form} identifies the real type decomposition
\[
\Omega^{(r,s+1)+(s+1,r)}(M)=\left[\Omega^{r,s}(M;TM)\oplus \Omega^{s+1,r-1}(M;TM)\right]\cap \Omega^{p+1}(M),
\]
with the $TM$-grading. When ${p=n}$, we get ${\Omega^{0,n}(M;TM)\cap \Omega^{n+1}(M) = \{0\}}$.
\end{lem}

\begin{proof}
Since both gradings decompose the entire space, it is enough to show an inclusion, which is a straightforward direct verification.
\end{proof}

\subsection{Traces}
We identify $A \in \End_{\mathbb{C}}(TM\otimes\mathbb{C})$ with the $(0,2)$-tensor\Marginpar{Important: use $g$, not $h$, so that $\phi$ is complex bilinear. So $A$ is $\phi$ up to the positioning of indices $\phi_{a}^{\enskip b} = A_a^b$.}
\begin{align}\label{Eqn:TensorEndomorphism}
	\phi_{XY}=g(A(X), Y).
\end{align}
Then $A$ is skew-Hermitian when $\phi$ is a $2$-form and complex linear for $J$ when $\phi$ is $J$-invariant. The (\emph{Lefschetz}) \emph{trace} of $\phi$ is\Marginpar{For $A$ a complexification $\Lambda(\phi)=\Im(\phi_\alpha^{\enskip\alpha}) \in \mathbb{R}$.}
\begin{align}\label{Eqn:LefschetzTrace}
\Lambda(\phi) = \frac12 h(\phi, F) = \frac{i}{2} \left(\phi^\alpha_{\enskip\alpha} - \phi^{\enskip\alpha}_\alpha\right)
\end{align}
using the inner product on $(0,2)$-tensors. This is just the Lefschetz trace of the anti-symmetrization of $\phi$. 
When $A$ is skew-Hermitian we thus have $\Lambda(\phi)=-i \phi_\alpha^{\enskip\alpha}$.

\noindent
The \emph{trace} of $A$ is given by
\begin{align}\label{Eqn:Trace}
\tr_\mathbb{C} A = A^\alpha_\alpha + A^{\bar\alpha}_{\bar\alpha} = \phi_\alpha^{\enskip\alpha} + \phi^\alpha_{\enskip\alpha}.
\end{align}
Then for $\tilde{A} = -J \circ A$ we have\Marginpar{$\phi^\alpha_{\enskip\alpha}=\frac12\tr_\mathbb{C} A - i \Lambda(\phi)$}
\begin{align}\label{Eqn:ComplexSummation}
\phi_\alpha^{\enskip\alpha} = \frac12\tr_\mathbb{C} A + i \Lambda(\phi) = \frac12\tr_\mathbb{C} A + \frac{i}{2}\tr_\mathbb{C} (\tilde{A})
\end{align}

Suppose $A$ is the complexification\Marginpar{Caution: when other complex indices $\beta$ remain, e.g.~$A_{\alpha\enskip\beta}^{\enskip\alpha}$, it is not a complexification.}
 of $a\in \End_\mathbb{R}(TM)$.\Marginpar{So when $\psi$ is a complexification the component of ordinary $(r,s)$-type is the conjugate of the component of ordinary $(s,r)$-type (i.e.~post compose with $T^{0,1}\cong T^{1,0}$).} Then $\overline{\phi_{ab}} = \phi_{\bar{a}\bar{b}}$ and the general fact $\phi^{\alpha}_{\enskip\alpha} = \phi_{\bar\alpha}^{\enskip\bar\alpha}$ imply that \eqref{Eqn:LefschetzTrace}, \eqref{Eqn:Trace} are real. Moreover, $\tr_\mathbb{C} A = \tr_\mathbb{R} a = a_i^{\;\, i}$. If $a$ is complex linear for $J$ then $\tr_\mathbb{R} a = 2 \tr_\mathbb{C} a$.


\subsection{Norms}

Let $E$ be a complex vector bundle on $M$ with Hermitian form $\langle,\rangle$. 
Generalizing the case $E=\C$, the norm of an $E$-valued differential $p$-form is\Marginpar{${\frac12|\psi|^2_{\otimes^3 T^*M} = |\psi|^2_{\Omega^2(M;TM)} = 3|\psi|^2_{\Omega^3(M)}}$}
\begin{align}\label{norms}
	|\psi|^2_{\Omega^p(M;E)} = \frac{1}{p!} \langle\psi_{i_1\ldots i_p}, \psi^{i_1\ldots i_p}\rangle = \frac{1}{p!} g^{i_1j_1}\cdots g^{i_pj_p} \langle \psi_{i_1\ldots i_p}, \psi_{j_1\ldots j_p}\rangle.
\end{align}
Unless $p=1$ we shall not follow \cite{MR1456265} in identifying $TM$-valued $p$-forms with ${(0,p+1)}$\nobreakdash-tensors, since this leads to different conventions for the norm. 
We will only need \eqref{norms} in the cases $E=\C$ and $E=TM$.
When an $E$-valued $p$-form $\psi$ is decomposed as a sum of elements \eqref{ordinarytype} then\Marginpar{Summing over all $r+s=p$. Note that indices are raised using the metric on $M$, while $\langle,\rangle$ is the metric on $E$}
\begin{align}\label{PQnorm}
|\psi|^2_{\Omega^p(M;E)}  = \frac{1}{r!s!} \langle\psi_{\alpha_1\ldots\alpha_r,\bar\beta_1\ldots\bar\beta_s},\psi^{\bar\alpha_1\ldots\bar\alpha_r,\beta_1\ldots\beta_s}\rangle.
\end{align}
In particular, the decomposition into $E$-type is orthogonal.

\begin{lem}\label{lem:comparisionNorms}
Let $\phi \in \Omega^{(r,r+1)+(r+1,r)}(M)$ be a real form with associated $TM$-valued form $\psi=i(\phi)$ using \eqref{pForm_TM_Form}. For the $\Omega^{2r}(M;TM)$-norm of the projections we have
\begin{equation}\label{comparisionNorms}
\frac{1}{r!r!}|\psi^{r+1,r-1}|^2 = \frac{1}{(r-1)!(r+1)!}|\psi^{r,r}|^2.
\end{equation}
When $n=2$, we have $2|\psi^{2,0}|=|\psi^{1,1}|^2$ for \emph{every} $3$-form $\psi$ \textup(see also \cite[(1.3.9)]{MR1456265}\textup).
\end{lem}
%

\subsection{Chern Connection}

The almost complex structure is parallel for the Levi-Civita connection $D^g$ precisely when $M$ is K\"ahler. Therefore one considers other metric connections that make $J$ parallel.

\begin{defn}
The \emph{Chern connection} $\nabla$ is the unique Hermitian connection on $TM$ whose $(0,1)$-part is the canonical Cauchy--Riemann operator
\begin{align}\label{Eqn:CauchyRiemann}
\bar\partial_X Z &= [X^{0,1}, Z]^{1,0},\qquad X\in TM, Z\in C^\infty(M,T^{1,0}).
\end{align}
(recall that a Hermitian connection is required to satisfy $\nabla g = 0, \nabla J = 0$.)
\end{defn}

Equivalently, the Chern connection is the unique Hermitian connection whose torsion tensor $T_{XY}=\nabla_X Y-\nabla_Y X-[X,Y]$ is $J$-anti-invariant. The decomposition of $T\in \Omega^2(M;TM)$ into $TM$-type is then given by (see~\cite[p.~272]{MR1456265})
\begin{align}\label{eqn:torsion}
	T^{0,2}&= N,
	&T^{1,1}&=0,
	&T^{2,0}&=(d^cF)^{2,0}.
\end{align}
Here $d^cF$ is a $TM$-valued $2$-form via \eqref{pForm_TM_Form} and we take the $(2,0)$-part of its $TM$-type.

\begin{rem}\label{kaehlercase}
If $T=0$ for the torsion of the Chern connection of an almost Hermitian manifold, then $\nabla=D^g$. Hence $J$ is parallel for $D^g$ and the structure is K\"ahler.
\end{rem}

\subsection{Torsion $1$-Form}\label{ssec:torsion}

Besides not being integrable, the difficulty in dealing with almost Hermitian manifolds is that the fundamental form is not closed ($dF=0$ holds when $M$ is \emph{almost K\"ahler}). We thus consider the \emph{torsion $1$-form}
\begin{align}\label{thetaDef}
\theta = \Lambda(dF) = J\delta F.
\end{align}
The almost Hermitian structure is \emph{Gauduchon} if $\delta^g \theta = 0$, and is \emph{balanced} if $\theta=0$. 

\Marginpar{Then $\rho$ is exact and $\int_M s^C \vol =\int_M \rho\wedge \frac{F^{n-1}}{(n-1)!}= \langle \rho, F\rangle_{L^2}$ and $F$ is co-closed.}

It is easy to check that the torsion $1$-form $\theta_X=\tr(Z\mapsto T_{XZ})$ is the trace of the torsion tensor of the Chern connection. Thus
\begin{align}
\theta = T_{\alpha\beta}^{\quad\beta} z^\alpha + T_{\bar\alpha\bar\beta}^{\quad\bar\beta} \bar{z}^{\bar\alpha}.
\end{align}
An equivalent definition is $dF=(dF)_0 + \frac{1}{n-1}\theta\wedge F$, for the trace-free part $(dF)_0$.


\subsection{Ricci Forms}

Let $R$ be a $2$-form with values in skew-Hermitian endomorphisms of $TM$, for example the curvature tensor $R^\nabla_{XY}=[\nabla_X,\nabla_Y]-\nabla_{[X,Y]}$ of the Chern connection. In the integrable case, the $2$-form $R^\nabla$ is $J$-invariant.
In general, the complexification of $R$ is not of type $(1,1)$ and has more components\Marginpar{Warning: some authors use the notation $g(R_{e_i,e_j}e_k,e_l)=R_{lkij}$. We dont.}
\begin{align*}
R=\left(\frac12R_{\alpha\beta\gamma}^{\;\;\;\;\;\;\delta}\, z^\alpha\wedge z^\beta +
R_{\alpha\bar\beta\gamma}^{\;\;\;\;\;\;\delta}\, z^\alpha\wedge \bar{z}^{\bar{\beta}}+
\frac12R_{\bar\alpha\bar\beta\gamma}^{\;\;\;\;\;\;\delta}\, \bar{z}^{\bar\alpha}\wedge \bar{z}^{\bar\beta}\right)\otimes z^\gamma \otimes z_\delta.
\end{align*}

Following \cite{MR742896}, we consider three ways to contract the tensor $R$:

\begin{defn}
The \emph{first} (or \emph{Hermitian}) \emph{Ricci form} $\rho$ or $R$ is the trace
\Marginpar{The trace of the complexification of $J\circ R_{XY}$ is twice of this.}\Marginpar{In Gauduchon, a minus sign is needed to match his opposite convention for the curvature tensor. In the integrable case, our Ricci forms and scalar curvatures agree with those of Gauduchon.}
\begin{align}
	\rho_{XY}=\tr_\mathbb{C} \left(J\circ R_{XY}\right) = -\Lambda(R_{XY}).
\end{align}
The complexification of $\rho$ has components
\begin{align}\label{firstRicci}
\rho = \frac{i}{2} R\indices{_{\alpha\beta\gamma}^\gamma} z^\alpha\wedge z^\beta
+i R\indices{_{\alpha\bar\beta\gamma}^\gamma} z^\alpha\wedge \bar{z}^{\bar\beta}
+\frac{i}{2} R\indices{_{\bar\alpha\bar\beta\gamma}^\gamma} \bar{z}^{\bar\alpha}\wedge \bar{z}^{\bar\beta}.
\end{align}
\end{defn}

The first Ricci form is always closed and, when $J$ is integrable, is of type $(1,1)$.
Its cohomology class $2\pi c_1(TM,J)$ is the first Chern class of $M$.

\begin{defn}
The \emph{second Ricci form} of $R$ is $r_{XY}=-\Lambda(R_{\,\cdot,\cdot,XY})$, so
\begin{align}\label{secondRicci}
r =  iR\indices{_\gamma^\gamma_{\lambda\bar\mu}} z^\lambda\wedge\bar{z}^{\bar\mu},
\end{align}
which is always a $(1,1)$-form, but not closed in general.\Marginpar{Even when $J$ is integrable}
\end{defn}

\begin{defn}
The \emph{third Ricci form} is
\begin{align}\label{thirdRicci}
\sigma = 
\frac{i}{2}R\indices{_{\mu\alpha}^\alpha_\lambda} z^\lambda\wedge z^\mu
+
iR\indices{_{\bar\mu\alpha}^\alpha_\lambda} z^\lambda\wedge \bar{z}^{\bar\mu}
-
\frac{i}{2}R\indices{_{\bar\mu}^\alpha_{\alpha\bar\lambda}} \bar{z}^{\bar\lambda}\wedge \bar{z}^{\bar\mu}
\end{align}
\end{defn}

\subsection{Scalar Curvatures} The Lefschetz traces of $\rho$ and $r$ agree. We thus define:

\begin{defn}
The \emph{Chern scalar curvature} of $R$ is
\begin{align}
s^C &= \Lambda(\rho) = \Lambda(r) = R_{\alpha\enskip\gamma}^{\enskip\alpha\enskip\gamma} = -\frac14  R_{e_i\enskip\;\, e_j}^{\enskip Je_i\enskip Je_j}
\end{align}
\end{defn}

The \emph{Hermitian scalar curvature} is $s^H=2\cdot s^C$ and coincides with the Riemannian scalar curvature in the K\"ahler case. The case $c_1(M)=0$ implies for instance that the total Hermitian scalar curvature vanishes.

\Note{
There are two approaches to convert complex into real coordinates, for example ${R_{\alpha\enskip\gamma}^{\enskip\alpha\enskip\gamma} = -\frac14  R_{e_i\enskip\;\, e_j}^{\enskip Je_i\enskip Je_j}}$. The first is based on \eqref{Eqn:Trace}, so $A_i^i = A_\alpha^\alpha + A_{\bar\alpha}^{\bar\alpha}$ when $A$ is a complexification. This gives
\begin{align}
R_{e_i\enskip e_j}^{\enskip Je_i\enskip Je_j} = R_{\alpha\enskip\beta}^{\enskip i\alpha\enskip i\beta}
+R_{\alpha\enskip\;\bar\beta}^{\enskip i\alpha\enskip -i\bar\beta}
+R_{\bar\alpha\enskip\enskip\enskip\beta}^{\enskip -i\bar\alpha\enskip i\beta}
+R_{\bar\alpha\enskip\enskip\enskip\bar\beta}^{\enskip -i\bar\alpha\enskip -i\bar\beta}
= -4 R_{\alpha\enskip\beta}^{\enskip\alpha\enskip\beta}
\end{align}
using $R_{\bar\alpha}^{\enskip\bar\alpha} = R_{\enskip\alpha}^{\alpha}=-R_\alpha^{\enskip\alpha}$. The second method uses \eqref{Eqn:ComplexSummation}. Define $b_{XY}=\frac12\tr R_{XY}$ and $c_{XY}=\Lambda(R_{XY})$ with complexifications $B,C$. Let $A_{\alpha\beta}=R_{\alpha\beta\gamma}^{\enskip\enskip\enskip\;\gamma}$. Then $A_{\alpha\beta} = B_{\alpha\beta} + iC_{\alpha\beta}$, according to \eqref{Eqn:ComplexSummation}. Hence
\begin{align*}
R_{\alpha\enskip\gamma}^{\enskip\alpha\enskip\gamma} \overset{\eqref{Eqn:ComplexSummation}}{=}& \frac12\tr (R_{\cdot\cdot\gamma}^{\enskip\enskip\gamma}) + i\Lambda(R_{\cdot\cdot\gamma}^{\enskip\enskip\gamma}) = \frac12\tr(A) + i\Lambda(A)
\end{align*}
which using $\tr(B)=\tr(b)$, $\tr(C)=\tr c$ becomes
\begin{align*}
R_{\alpha\enskip\gamma}^{\enskip\alpha\enskip\gamma}=\frac12\tr(b) + \frac{i}{2}\tr(c) + i\Lambda(b) - \Lambda(c)
\end{align*}
Now $\tr(b)=\tr(c)=\Lambda(b)=0$ using the skew symmetry of $R$ and
\[
	\Lambda(c) = \frac12 c_{e_i,Je_i} = \frac12\Lambda(R_{e_i,Je_i}) = \frac14 R_{e_i,Je_i,e_j,Je_j}
\]
}

\begin{defn}
 The \emph{third scalar curvature} is the Lefschetz trace
 \begin{align}\label{Eqn:ThirdScalarCurvature}
 s = \Lambda(\sigma) = R_{\alpha\enskip\beta}^{\enskip\beta\enskip\alpha}
=\frac12 R_{i\;\,j}^{\;\,j\;\,i} = -\frac12 R_{ij}^{\;\,\;\,ij}
 \end{align}
of the third Ricci form\Marginpar{real-valued}. Alternatively, $s$ is the 
trace of the curvature operator.
\end{defn}

We shall also consider the \emph{Riemannian scalar curvature} $s^g$ formed as usual from the Levi-Civita connection $D^g$.

\begin{rem}
Suppose $J$ is integrable. Then the three Ricci forms coincide with those of \cite[Section I.4]{MR742896}. Gauduchon uses the notation $u$ for $s^C$ and $v$ for $s$. Liu--Yang define similar scalar curvatures from the Levi-Civita connection. The third scalar curvature $s$ is like their `Riemannian type scalar curvature' $s_R$ in \cite[4.2]{Liu-Yang}.
\end{rem}

\section{Comparison of Curvatures}\label{sec:ComparisionScalar}

When $M$ is not K\"ahler, the three scalar curvatures defined above do not coincide. In this section, we quantify the differences. 
This is based on the following.
Recall that the algebraic Bianchi identity for connections on $TM$ with torsion asserts
\begin{equation}
\begin{aligned}\label{algebraicBianchi}
R_{XY}Z + R_{ZX}Y + R_{YZ}X = (&\nabla_XT)_{YZ}+(\nabla_ZT)_{XY}+(\nabla_YT)_{ZX}\\
-& T(X,T_{YZ}) -T(Z,T_{XY}) -T(Y,T_{ZX}).
\end{aligned}
\end{equation}
Also, by \cite[(2.1.4)]{MR742896} the difference between the Chern and Levi-Civita connection is
\begin{align}\label{Eqn:ChernLeviCivita}
g(\nabla_X Y,Z) = g(D^g_X Y,Z) + \frac32 t_{XYZ} - g(X,T_{YZ}),
\end{align}
for the anti-symmetrization
\begin{align}
t_{XYZ}= \frac13\left(g(X,T_{YZ}) + g(Z,T_{XY}) + g(Y,T_{ZX})\right)
\end{align}
of the torsion tensor. Formula \eqref{Eqn:ChernLeviCivita} is also easily deduced from Lemma~\ref{Lem:ChernConnection} below. As a consequence of~\eqref{eqn:torsion} we note (see also~\cite[(2.5.10)]{MR1456265})
\begin{align}\label{tdcF}
t=\frac13 d^cF.
\end{align}

\begin{prop}\label{prop:ScalarComparision}
We have\Marginpar{In your notes, $s^\nabla$ is what we call $2s$ and same for $s^C$. Moreover, you use the $\otimes^3 T^*M$ tensorial norm for $T$, which is twice of our $\Omega^2(M;TM)$-norm. Thus, our formulas agree.} $s^C - s = \frac12|\theta|^2+\frac12\delta^g\theta - \frac92|t|^2+\frac12|T|^2$.
\end{prop}

Here, $|T|^2$ is given by \eqref{norms} as a $TM$-valued $2$-form and $|\theta|^2$ and $|t|^2$ are the usual norms for differential $3$-forms. 
The codifferential $\delta^g$ is taken with respect to $g$.

\begin{proof}
Since $\nabla(J)=0$, \eqref{algebraicBianchi} reduces in complex coordinates to
\begin{align}\label{Bianchi}
R_{\alpha\bar\beta\gamma}^{\enskip\quad\delta} + R_{\bar\beta\gamma\alpha}^{\quad\enskip\delta} &= (\nabla_{\bar\beta} T)_{\gamma\alpha}^{\quad\delta} - T_{\bar\beta,T(\gamma,\alpha)}^{\quad\qquad\delta}.
\end{align}
(use that $\nabla$ preserves the type decomposition of $TM$-valued forms.)

Choose an orthonormal frame $e_{2\alpha-1}, e_{2\alpha}=Je_{2\alpha-1}$ near $p\in M$ with ${\nabla_X e_i = 0}$ for all $X\in T_p M$. Such a frame may be constructed by extending parallelly an adapted orthonormal basis in $p$ along geodesic rays. We work in the basis ${z_\alpha=\frac12(e_{2\alpha-1}-ie_{2\alpha})}$ of $T^{1,0}$.
Evaluating \eqref{Eqn:ChernLeviCivita} at $p$ gives
\begin{align*}
D^g_\alpha (z_{\bar\alpha}) = D^g_{\bar\alpha}(z_\alpha) &= -\frac14\theta^\beta z_\beta - \frac14 \theta^{\bar\beta} z_{\bar\beta}.
\end{align*}
so at $p$ the codifferential reduces to
\begin{equation*}
\begin{aligned}
\frac12 \delta^g\theta &= -\bar{z}_{\bar\alpha} \theta_\alpha - z_\alpha \theta_{\bar\alpha} - \frac12|\theta|^2.
\end{aligned}
\end{equation*}
Taking the double trace over \eqref{Bianchi} we obtain
\begin{align*}
s^C-s=R_{\alpha\enskip\gamma}^{\enskip\alpha\enskip\gamma} + R^{\alpha\quad\gamma}_{\enskip\gamma\alpha} &= (\nabla_{\bar\alpha} T)_\gamma^{\enskip\bar\alpha\gamma} - T^{\alpha\qquad\;\gamma}_{\enskip T(\gamma,\alpha)}
= (\nabla_\alpha T)_{\bar\gamma}^{\enskip\alpha\bar\gamma} - T^{\bar\alpha\qquad\;\bar\gamma}_{\enskip T(\bar\gamma,\bar\alpha)}.
\end{align*}
(the last equation holds since $s^C-s$ is real.) In our frame
\begin{align*}
\frac12\big((\nabla_{\bar\alpha} T)_\gamma^{\enskip\bar\alpha\gamma}
+(\nabla_{\alpha} T)_\gamma^{\enskip\alpha\bar\gamma}\big) = -\bar{z}_{\bar\alpha}\theta_\alpha - z_\alpha \theta_{\bar\alpha}
\end{align*}
at the point $p$. Thus
\begin{align*}
s^C - s = \frac12 (|\theta|^2 + \delta\theta) - \frac12\left( T^{\alpha\qquad\;\gamma}_{\enskip T(\gamma,\alpha)}+ T^{\bar\alpha\qquad\;\bar\gamma}_{\enskip T(\bar\gamma,\bar\alpha)}\right).
\end{align*}
Now apply the easy identities $T^{\alpha\qquad\;\gamma}_{\enskip T(\gamma,\alpha)} = T^{\bar\alpha\qquad\;\bar\gamma}_{\enskip T(\bar\gamma,\bar\alpha)} = T_{\alpha\beta\gamma} T^{\beta\gamma\alpha}$ and
\begin{equation}\label{tTnorm}
9|t|^2-|T|^2 = 2T_{\alpha\beta\gamma}T^{\beta\gamma\alpha} = T_{ijk}T^{jki}.\qedhere
\end{equation}
\end{proof}

\begin{prop}\label{prop:ScalarComparision1}
For the Riemannian scalar curvature $s^g$ we have\Marginpar{I now use the $\Omega^2(M;TM)$-norm, instead of the $\otimes^3$-norm used by you and Gauduchon}
\begin{align}\label{ssg}
2s-s^g = |T|^2-\frac92|t|^2-2\delta\theta - |\theta|^2 
\end{align}
\end{prop}

\begin{proof}
This is a similar computation, using normal coordinates $e_i$ at $p\in M$. Thus $[e_i,e_j]=0$ and $\left.D^g_{e_i}(e_j)\right|_p=0$. The Riemannian scalar curvature at $p$ is then
\begin{align*}
s^g(p) = e_ig(D^g_{e_j}e_j,e_i) - e_jg(D^g_{e_i}e_j,e_i)
\end{align*}
(we omit all summation signs) which, using
\begin{align}\label{ChernLC1}
\nabla_{e_i}e_j = D^g_{e_i}(e_j) + \left(\frac32t_{ij}^{\enskip\; k} - T_{j\enskip i}^{\,\, k}\right)e_k
\end{align}
from \eqref{Eqn:ChernLeviCivita}, becomes
\begin{align*}
&e_ig(\nabla_{e_j}e_j,e_i) + e_iT_{jij} - e_jg(\nabla_{e_i}e_j,e_i) - e_jT_{jii}\\
=& 2s(p) + g(\nabla_{e_i}e_i, \nabla_{e_j}e_j) - g(\nabla_{e_i}e_j, \nabla_{e_j}e_i)  - 2e_jT_{jii}.
\end{align*}
Now $\delta\theta(p) = - e_jT_{jii}$ and inserting \eqref{ChernLC1} gives
\begin{align*}
s^g(p)-2s(p) &= T_{iki}T_{jkj} + \frac94t_{ijk}^2 + 3t_{ij}^{\;\;k}T_{i\,\,j}^{\,\,k} - T_{jki}T_{ikj} + 2\delta\theta
\end{align*}
Now apply \eqref{tTnorm} and $|\theta|^2 = T_{iki}T_{jkj}$ to get \eqref{ssg}.
\end{proof}

Combining Propositions~\ref{prop:ScalarComparision} and \ref{prop:ScalarComparision1} gives (recall $s^H=2s^C$):
\begin{cor}\label{sHsG}
$s^H - s^g = -\delta\theta - \frac{27}{2}|t|^2 + 2|T|^2$.
\end{cor}

\begin{rem}\label{rem:intCaseScalar}
When $J$ is integrable, we have $9|t|^2=|T|^2$. In this case \eqref{ssg} reduces to \textup{\cite[(32)]{MR742896}} and Proposition~\ref{prop:ScalarComparision} specializes to \cite[Corollaire~2]{MR742896}:
\begin{align*}
2s - s^g &= \frac12|dF|^2 - 2\delta\theta - |\theta|^2\\ 
s^H - 2s &= |\theta|^2 + \delta\theta\\
s^H - s^g &= \frac12|dF|^2 - \delta\theta
\end{align*}
(in this case $dF$ is of type $(2,1)+(1,2)$ and so $|(d^cF)^{2,0}|_{\Omega^2}^2 = |dF|_{\Omega^3}^2$ by Lemma~\ref{lem:comparisionNorms}.)
\end{rem}

\section{Integrability Theorems}\label{sec:VanishingTheorems}

In the almost K\"ahler case, $dF=0$, Propositions~\ref{prop:ScalarComparision} and \ref{prop:ScalarComparision1} immediate imply vanishing theorems. For in this case, $T=N$, $\theta=0$, and $t=0$ from \eqref{eqn:torsion}, \eqref{thetaDef}, and \eqref{tdcF}, respectively. The formulas then reduce to
\begin{align}\label{almostKaehlerScalar}
2s-s^g&=|N|^2,
& s^H - 2s = |N|^2.
\end{align}
\begin{cor}\cite{apostolov1999almost}\label{cor:VanishAK}
On an almost K\"ahler manifold we have $s^g \leq 2s \leq s^H$ with either equality precisely when $(J,g,F)$ is K\"ahler.
\end{cor}

With some care in dimension four, these conclusions can be extended.
Thus, assuming equality of various scalar curvatures on an almost Hermitian manifold will guarantee both the
integrability af $J$ and the K\"ahler condition $dF=0$.

\begin{thm}\label{THM:vanishing}
Let $(M,J,g,F)$ be a closed almost Hermitian $4=2n$-manifold.
\begin{enumerate}
\item
$\displaystyle\int_M (2s-s^g + |\theta|^2)\frac{F^n}{n!} \geq 0$.
\item
$\displaystyle\int_M (s^H - s^g)\frac{F^n}{n!} \geq 0$.
\item
$\displaystyle\int_M (s^C - s)\frac{F^n}{n!} \geq 0$.
\end{enumerate}
In any case, equality holds if and only if the structure is K\"ahler.
\end{thm}

\begin{proof}
Recall from \eqref{tdcF} and \eqref{eqn:torsion} that
\[
	T=N+(d^cF)^{2,0},\qquad
	t = \frac13 d^cF.
\]
Since we are in dimension four, $t=t^{2,0}+t^{1,1}$ by Lemma~\ref{lem:realTMtype}. 
Also for the $\Omega^2(M;TM)$-norm defined in \eqref{norms}, Lemma~\ref{lem:comparisionNorms} gives
\begin{equation}\label{tNorms}
	|t|^2_{\Omega^2}=|t^{2,0}|^2 + |t^{1,1}|^2 = 3|t^{2,0}|^2 = \frac13|(d^cF)^{2,0}|^2.
\end{equation}
Combined with $|t|^2_{\Omega^3} = \frac13|t|^2_{\Omega^2(M;TM)}$ we get
\begin{equation}\label{better4}
	|T|_{\Omega^2}^2 - \frac92|t|_{\Omega^3}^2 = |N|^2 + \frac12|(d^cF)^{2,0}|^2.
\end{equation}
Putting this into Proposition~\ref{prop:ScalarComparision1} and integrating gives
\[
\int_M (2s-s^g + |\theta|^2) = \int_M \left(|N|^2 + \frac12|(d^cF)^{2,0}|^2 \right) \frac{F^n}{n!} \geq 0
\]
where we use that the integral over $\delta\theta$ vanishes (since $\frac{F^n}{n!}$ is the Riemannian volume form). Of course, the left hand side can only vanish when $N=0$ and $(d^cF)^{2,0}=0$. Then $T=N+(d^cF)^{2,0} = 0$ so by Remark~\ref{kaehlercase} we are in the K\"ahler case.
Part ii) is a similar application of Corollary~\ref{sHsG}, while iii) uses Proposition~\ref{prop:ScalarComparision}.
\end{proof}

\begin{rem}
When $M$ is a closed Hermitian manifold (the integrable case), one can
deduce Theorem~\ref{THM:vanishing} in any dimension (see~\cite{MR742896, Liu-Yang} or apply the technique above to Remark~\ref{rem:intCaseScalar}). On the other hand, Dabkowski--Lock~\cite{Lock-Dab} have examples of \emph{non-compact} Hermitian with $s^H=s^g$ which are not K\"ahler. Do higher-dimensional closed almost Hermitian non-K\"ahler manifolds with $s^H=s^g$ exist? 
\end{rem}

In the conformally almost K\"ahler case, we will extend ii)~to higher dimensions in Theorem~\ref{ineq_chern_riem_con_AK} below. We now proceed by proving an `opposite' of Corollary~\ref{cor:VanishAK} for nearly K\"ahler structures.

\begin{defn}
An almost Hermitian manifold $(J,g,F)$ is {\it nearly K\"ahler} if
\[
\left(D^g_XJ\right)Y+\left(D^g_YJ\right)X=0
\]
where $D^g$ is the Levi-Civita connection~\cite{gray1966some,gray1970nearly}.
\end{defn}

It follows from the definition that $D^gF=\frac{1}{3}dF$. Moreover, $dF$ is of type $(3,0)+(0,3)$
and $N=\frac{1}{3}d^cF$ is totally anti-symmetric. In particular, a nearly K\"ahler manifold is balanced.
Furthermore, a nearly K\"ahler manifold of dimension $2n=4$ is K\"ahler. Also, if the nearly K\"ahler manifold is
Hermitian then it is K\"ahler. Examples of nearly K\"ahler manifolds
are $S^6$ with its standard almost-complex structure and metric, $S^3\times S^3$
equipped with the bi-invariant almost complex structure and its 3-symmetric almost Hermitian structure
and the twistor spaces over Einstein self-dual 4-manifolds, endowed with the anti-tautological almost complex structure (for more details about nearly K\"ahler manifolds see~\cite{nagy2002nearly,butruille2005classification,moroianu2005unit,verbitsky2011hodge,butruille2010homogeneous}).
We deduce from Propositions~\ref{prop:ScalarComparision} and \ref{prop:ScalarComparision1} 
that for any nearly K\"ahler manifold
\begin{equation*}
2s-s^g = -\frac16|d^cF|^2_{\Omega^3},\qquad s^H-2s = -\frac23|d^cF|^2_{\Omega^3}.
\end{equation*}


\begin{cor}\label{nearlyKaehlerVanish}
On a nearly K\"ahler manifold we have $s^H \leq 2s \leq s^g$ with either equality precisely when $(J,g,F)$ is K\"ahler.
\end{cor}

\begin{rem}
This does not contradict Theorem~\ref{THM:vanishing} because in dimension $4$ the notions K\"ahler and nearly K\"ahler agree.
\end{rem}

A feature of nearly K\"ahler manifolds is that $dF$ is of constant norm~\cite{kirichenko1977k,butruille2005varietes}. 
Moreover, Gray proved~\cite[Theorem 5.2]{gray1976structure} that any non-K\"ahler nearly K\"ahler manifold of dimension $6$
is Einstein of positive (constant) $s^g$. Furthermore, their first Chern class vanishes. Hence, the Hermitian scalar curvature $s^H$ of any closed non-K\"ahler nearly K\"ahler manifold of dimension $6$ vanishes.

\begin{cor}\label{curious-cor}
On a closed non-K\"ahler, nearly K\"ahler manifold $(M,J,g)$ of dimension $6$, we have $s^H=0$.
\end{cor}

\section{Conformal Variations}\label{sec:ConformalChanges}

Let $\tilde{g} = e^{2f}g$ be a conformal variation of the metric along a smooth real-valued function $f$. Then $(M,J,\tilde{g},\tilde{F})$ is again an almost Hermitian manifold and we shall be interested in how the associated Chern connection, Ricci forms, and scalar curvatures behave under this variation.

We begin by deriving an alternative expression for the Chern connection:

\begin{lem}\label{Lem:ChernConnection}
The Chern connection is given by
\begin{align}\label{Eqn:ChernConnection}
h(W,\nabla_X Z) &= X^{0,1}h(W,Z) + h(W,[X^{0,1},Z]) + h([W,X^{0,1}],Z),
\end{align}
where $X\in TM$ and $W,Z \in C^\infty(M,T^{1,0})$.
\end{lem}

\begin{proof}
\eqref{Eqn:ChernConnection} is easily seen to define a Hermitian connection whose $(0,1)$-part is given by \eqref{Eqn:CauchyRiemann}. The result then follows from the uniqueness of such a connection.
\end{proof}

\begin{lem}\label{Lem:ConformalChern}
$\tilde{\nabla}_X Z = \nabla_X Z + 2X^{1,0}(f)\cdot Z$ for all $X\in TM, Z \in T^{1,0}$.
\end{lem}

\begin{proof}
This is an immediate consequence of Lemma~\ref{Lem:ChernConnection}.
\end{proof}

\begin{prop}\label{Prop:ChernCurvature} For the curvature tensors of the Chern connection we have
\begin{align}
\tilde{R}^{\tilde{\nabla}}(Z) = R^\nabla(Z) + i dd^cf \cdot Z.
\end{align}
\end{prop}

\begin{proof}
Beginning with Lemma~\ref{Lem:ConformalChern} a straightforward calculation gives
\[
\tilde{R}^{\tilde\nabla}_{XY}Z = R^\nabla_{XY}Z + 2\left( X(Y^{1,0}f)-Y(X^{1,0}f)-[X,Y]^{1,0}f\right)\cdot Z.\qedhere
\]
\end{proof}

\begin{cor}\label{RicciVar}
The conformal variations of the three Ricci forms are given by
\begin{align}
\tilde{\rho} &= \rho - n\cdot dd^cf,\label{Eqn:FirstRicciConformal}\\
\tilde{r} &= r - \Lambda(d d^cf)\cdot F,\label{Eqn:SecondRicciConformal}\\
\tilde{\sigma} &= \sigma - dd^cf.\label{Eqn:ThirdRicciConformal}
\end{align}
\end{cor}
\begin{proof}
Compute the three Ricci forms of the tensor $dd^cf\otimes F$.
\end{proof}

\begin{cor}\label{cor:ConformalScalVar}
The conformal variations of the scalar curvatures are\Marginpar{$\tilde{\Lambda} = e^{-2f} \Lambda$}
\begin{align}
e^{2f}\tilde{s}^C &= s^C - n \Lambda(dd^cf),\label{Eqn:Scal}\\
e^{2f}\tilde{s} &= s - \Lambda(dd^cf).\label{Eqn:ThirdScal}
\end{align}
\end{cor}

\begin{lem}
For $f\in C^\infty(M)$ real we have\Marginpar{${e^{2f}\tilde{s}^H = s^H + m\Delta^g(f) + mg(\theta,df)}$ for ${\tilde{g}=e^{2f}g}$. Recall $s^H = 2s^C$}
\begin{align}\label{ChernLaplacian}
-\Lambda(dd^cf)=\Delta^g(f) + g(\theta,df).
\end{align}
Here, $\Delta^g$ denotes the Hodge--de~Rham operator $\Delta^g(f)=\delta^g df$.
\end{lem}
\begin{proof}
Let $(e_i)_{i=1,\ldots,m}$ by a $J$-adapted orthonormal frame as in \eqref{Jadapted}. Then
\begin{align*}
\Lambda(dd^cf)=&\frac{1}{2}\sum_{i=1}^m\left(dd^cf\right)(e_i,Je_i)\\
=&\sum_{i=1}^m e_i\left(d^cf(Je_i)\right)-d^cf(JD^g_{e_i}e_i)-\left(d^cf\right)\left((D^g_{e_i}J)e_i\right)\\
=&-\Delta^g(f) - g(\theta,df).
\end{align*}
In the last equality we have used $\theta=J\delta^gF=-\displaystyle\sum_{i=1}^mg(J(D^g_{e_i}J)e_i,\cdot)$, which is a straightforward consequence of $D^gg=0$ and $F=g(J\cdot,\cdot)$.
\end{proof}

\begin{cor}
For the Hermitian scalar curvature of $\tilde{g}=e^{2f}g=u^{-2}g$ we have
\begin{align}
e^{2f}\tilde{s}^H &= s^H + m\Delta^g(f) + mg(\theta,df)\label{Eqn:HermScal}\\
\tilde{s}^H &= u^2s^H - mu\left(\Delta^g(u) + g(\theta, du)\right) - m|du|^2_g\label{Eqn:HermScalPoly}
\end{align}
\end{cor}

For \eqref{Eqn:HermScalPoly} use ${-e^f{\Delta e^{-f}} = \Delta f + |df|^2}$.
When $J$ is integrable we recover \cite[(23)]{MR742896}.
\Marginpar{
${\Delta \ln \varphi = \varphi^{-1}\Delta \varphi + \varphi^{-2}|d\varphi|^2}$,
${e^{-f}{\Delta e^f} = \Delta f - |df|^2}$
 so
for $\tilde{g}=\varphi^{2a}g$ we have
$
	\varphi^{2a}\tilde{s}^C = s^C + \frac{na}{\varphi}\big(\Delta\varphi + g(\theta,d\varphi)\big) + \frac{na}{\varphi^2}|d\varphi|^2
$
}

\begin{thm}\label{ineq_chern_riem_con_AK}
Let $(M,J,g,F)$ be a closed almost Hermitian manifold of real dimension $m=2n$.
Assume that $g$ is conformally almost K\"ahler. Then
\[
	\int_M (s^H - s^g) \frac{F^n}{n!} \geq 0.
\]
Equality holds precisely when $J$ is integrable and $(J,g,F)$ is K\"ahler.
\end{thm}

\begin{proof}
Suppose $(J,\tilde{g}=e^{2f}g)$ is almost K\"ahler. Then by \eqref{almostKaehlerScalar} we have
\[
	\tilde{s}^H - \tilde{s}^{\tilde{g}} = 2|N|_{\tilde{g}}^2 = 2e^{-2f}|N|_{g}^2
\]
Since $\tilde{F}=e^{2f}F$ is closed, we have
\[
	0=d(e^{2f}F) = e^{2f}\left( (dF)_0 + \left(2df + \frac{\theta}{n-1}\right)\wedge F \right)
\]
for the trace-free part $(dF)_0$. From this we read off the torsion $1$-form
\[
	\theta = (2-m)df.
\]
Putting this into \eqref{Eqn:HermScal} and combining with the formula for the conformal variation of the Riemannian scalar curvature (see Besse~\cite[Theorem~1.159]{MR867684}) we get
\begin{equation}\label{hopf_operator}
	e^{2f}\left( \tilde{s}^H - \tilde{s}^{\tilde{g}} \right) =
	(s^H-s^g) + (2-m)\Delta f - (m-2)|df|^2.
\end{equation}
Hence
\[
	\int_M 2|N|_g^2 \frac{F^n}{n!} = \int_M (s^H-s^g)\frac{F^n}{n!} - (m-2)\int_M |df|_g^2 \frac{F^n}{n!}.\qedhere
\]
\end{proof}

\section{Conformally Constant Metrics}\label{sec:conformally}

We shall be concerned with the existence of the following type of metrics:

\begin{defn}
An almost Hermitian metric $(J,g,F)$ is \emph{conformally constant} if for some $f\in C^\infty(M)$ the structure $(J,\tilde{g},\tilde{F})\coloneqq (J,e^{2f}g,e^{2f}F)$ has $\tilde{s}^H=\text{const}$.
\end{defn} 

In Corollary~\ref{corExtend} we prove a sufficient criterion for $(J,g,F)$ to be conformally constant (non-positive fundamental constant). This can be regarded as a generalization of the Chern--Yamabe problem~\cite{Ang-Cal-Spo} to the non-integrable case. Thus the problem is divided into the cases ${C(J,[g])\leq 0}$ and ${C(J,[g])>0}$ according to the fundamental constant. The positive case in the Chern--Yamabe problem is difficult because \eqref{Eqn:HermScal} looses its nice analytic properties stemming from the maximum principle. The question remains open in this case. Restricting to symplectic manifolds, we shall consider instead the following more basic existence problem:

\begin{EXproblem}\label{EXproblem}
Let $(M,\omega)$ be a closed symplectic manifold. Does $M$ admit \emph{any} almost complex structure $J$ such that $(J,g,\omega)$ is conformally constant?
\end{EXproblem}

\Marginpar{It suffices to study this problem in the special case where every almost Hermitian structure on $M$ has positive Gauduchon degree.}

\begin{rem}
We allow conformal variations because if we fix $\omega$ the existence of compatible metrics with $s^H=\text{const}$ is not guaranteed: sometimes one cannot find an {\it extremal} K\"ahler metric~\cite{calabi1985extremal} and for instance on toric manifolds the existence of extremal K\"ahler metrics is conjecturally equivalent to the existence of extremal almost-K\"ahler metrics (see~\cite{MR1988506} and also for example~\cite{MR2425136,MR2807093,Apostolov:2017aa,Keller:2017aa,levine1985remark,tian1997kahler,ross2006obstruction}).
\end{rem}

\begin{CYproblem}\label{CY}
Given a closed almost Hermitian manifold $M$, find a conformal structure $(J,e^{2f}g,e^{2f}F)$ of constant Hermitian scalar curvature. In other words, is every $(J,g,F)$ conformally constant?
\end{CYproblem}

As shown in \eqref{Eqn:HermScal}, the Hermitian scalar curvature transforms by the same formula as in the integrable case. Here we show how to extend the main results of~\cite{Ang-Cal-Spo} to the non-integrable case, as well as some results of independent interest. We mention also the work \cite{del2003yamabe} where a similar problem for the $J$-scalar curvature is studied, which is derived from the Riemannian curvature.

Recall that Gauduchon showed in~\cite{MR0470920} that every conformal class $[g]$ has a natural base-point $g_0=e^{-2f_0}g$. It is characterized by having a co-closed torsion $1$-form $\theta_0$, once we normalize $g_0$ to unit volume. In terms of the \emph{complex Laplacian}
\[
	L^g(f) \coloneqq \Delta^g f + g(\theta,df),
\]
this is equivalent to $(L^g)^*e^{(m-2)f_0} = 0$ and ${\int_M e^{-mf_0} \frac{F^n}{n!} = 1}$.

\begin{defn}
${(J,g_0\coloneqq e^{-2f_0}g,F_0\coloneqq e^{-2f_0}F)}$ is the \emph{Gauduchon metric} in the conformal class $[g]$.
The \emph{fundamental constant} is (see~\cite{Ang-Cal-Spo,MR1712115,MR779217,MR0486672})\Marginpar{If not normalized to unit volume, the correct expression is $\int_M e^{(2-m)f_0} s^H \frac{F^n}{n!} / \int_M e^{-mf_0} \frac{F^n}{n!}$. I.e.~for the unnormalized Gauduchon metric, one needs to divide by the total volume.}
\begin{equation}\label{GauduchonDegree}
	C(M,J,[g]) \coloneqq \int_M e^{(2-m)f_0} s^H \frac{F^n}{n!} \overset{\eqref{Eqn:HermScal}}{=} \int_M s_0^H \frac{F_0^n}{n!}.
\end{equation}
\end{defn}

%
In the Hermitian setting, the fundamental constant plays a central role in the Plurigenera Theorem~\cite{MR0486672}
and is closely related to the Kodaira dimension. The different cases in the Chern--Yamabe problem are $C<0$, $C=0$, and $C>0$.

Recall that the \emph{Yamabe constant} is defined in terms of the Riemannian structure
\[
Y[g]\coloneqq
\inf  \left\{\int_M s^{\tilde{g}} \vol_{\tilde{g}} \enskip\middle|\enskip \tilde{g}=e^{2f}g,\enskip \int_M \vol_{\tilde{g}} = 1\right\}.
\]
We remark that Yamabe, Trudinger, Aubin, and Schoen have shown that $[g]$ contains metrics of constant Riemannian scalar curvature $Y[g]$ (see \cite{MR888880} for a full account). From Theorem~\ref{THM:vanishing} we immediately get (see~\cite{MR1712115} in the integrable case):

\begin{prop}
In dimension $2n=4$ we have the estimate
\[
Y[g] \leq C(M^4,J,[g])
\]
with equality if and only if the Gauduchon metric $(J,g_0,F_0)$ is K\"ahler of constant scalar curvature.
\end{prop}

\begin{prop}\label{propVZ}
Let $(M^{m},J,g,F)$ be a closed almost Hermitian manifold. Then there exists a conformal metric $\tilde{g} \in [g]$ whose Hermitian scalar curvature has the same sign as $C$ at every point \textup(meaning zero when $C=0$\textup).
\end{prop}

\begin{proof}
The adjoint of the complex Laplacian $L^{g_0}$ of the Gauduchon metric $g_0$ is $\Delta^{g_0}f - g_0(\theta_0,df)$, where we use $\delta^{g_0}\theta_0=0$. By the maximum principle $\ker (L^{g_0})^*$ are the constant functions (for more details, see~\cite{MR0470920}). Hence the equation
\[
	L^{g_0}f = C(J,[g])-s^H_{g_0}
\]
is solvable for $f$, since the right hand side is orthogonal to the constants. Defining $\tilde{g}=e^{2f}g_0$, equation \eqref{Eqn:HermScal} shows $\tilde{s}^H = e^{-2f}C(J,[g])$.
\end{proof}

\begin{rem}\label{CzeroCase}
This generalizes \cite[Theorem 3.1]{Ang-Cal-Spo} to the non-integrable case. It follows that the Chern--Yamabe problem is solvable when $C=0$. The same conclusion (and same proof) holds for the third scalar curvature, where $C$ is replaced by the integral of the third scalar curvature of the Gauduchon metric $g_0$.
\end{rem}

A careful review of the analytic content of the argument given by Angella--Calamai--Spotti for \cite[Theorem~4.1]{Ang-Cal-Spo} reveals the following statement:

\begin{thm}[\cite{Ang-Cal-Spo}]\label{ACS}
Let $(M^{m},J,g,F)$ be a closed almost Hermitian manifold, and let ${S\colon M\to \mathbb{R}}$ be any strictly negative smooth function \textup(not necessarily the scalar curvature\textup). Then the PDE
\begin{equation}\label{PDE-ACS}
	mL^g(f) + S = \lambda e^{2f}
\end{equation}
has a solution $(\lambda, f) \in \mathbb{R} \times C^\infty(M)$; in fact we must have $\lambda<0$. The solution is unique up to replacing $(\lambda,f)$ by $(\lambda e^{-2c}, f+c)$ for a constant $c$. Thus by scaling we may solve \eqref{PDE-ACS} for any given negative $\lambda$.
\end{thm}

\begin{proof}
This is proven in \cite[p.~11]{Ang-Cal-Spo} by the continuity method. For their argument it is only important to see that for any solution $f$ of \eqref{PDE-ACS} we  have $\lambda<0$. In our general setting, this follows since putting the formula $L^g(f) = e^{-2f_0} L^{g_0}(f)$ into \eqref{PDE-ACS} and integrating gives
\[
	m \underset{=0}{\underbrace{\int_M \left(\Delta^{g_0}f + g_0(\theta_0,df)\right)\frac{F_0^n}{n!}}} + \underset{<0}{\underbrace{\int_M Se^{2f_0} \frac{F_0^n}{n!}}} = \lambda \underset{>0}{\underbrace{\int_M e^{2(f+f_0)}\frac{F_0^n}{n!}}}.\qedhere
\]
\end{proof}

%
%

Combining this with Proposition~\ref{propVZ} and \eqref{Eqn:HermScal} we thus obtain the following generalization of \cite[Theorem~4.1]{Ang-Cal-Spo}:

\begin{cor}\label{corExtend}
Every closed almost Hermitian manifold with ${C(J,[g])\leq 0}$ is conformally constant \textup(see~also Remark~\textup{\ref{CzeroCase}}\textup).
\end{cor}

\begin{rem}
We refer also to Berger~\cite{MR0301680} for a related question. When $(J,g)$ is K\"ahler (or more generally when $[g]$ is a balanced conformal class) he essentially constructs
solutions of \eqref{PDE-ACS} where $S$ is the Hermitian scalar curvature of $(J,g)$ and $\lambda$ is a given non-positive function. 
\end{rem}


\section{Ruled Manifolds}\label{sec:BabyCase}

We begin our study of the Existence Problem~\ref{EXproblem} for positive fundamental constant with ruled manifolds.
On complex manifolds, $C(J,[g])>0$ implies Kodaira dimension $-\infty$, by the Gauduchon Plurigenera Theorem~\cite{MR0486672}. The Kodaira dimension of ruled manifolds is $-\infty$ (conversely, this however does not imply $C(J,[g])>0$).

Angella--Calamai--Spotti~\cite[Section 5]{Ang-Cal-Spo} have given first simple examples of Hermitian non-K\"ahler manifolds of positive constant Hermitian scalar curvature (for instance on the Hopf surface or abstractly by deformations using the implicit function theorem).
In this section, we demonstrate the existence of almost Hermitian non-K\"ahler metrics of positive constant Hermitian scalar curvature on ruled manifolds (see~\cite{MR1647571} in the case of {\it{extremal}} K\"ahler metrics). We mention also Hong's work \cite{hong1999constant}, which is different in that only the K\"ahler \emph{class} is fixed.

\subsection{The generalized Calabi construction}

Let us briefly review the construction. The reader may consult~\cite{MR2144249,MR2425136,MR2807093} for more details and greater generality.

Let $(S,\omega_S)$ be a symplectic manifold.
For a torus $T$ with Lie algebra $\mathfrak{t}$ let $\pi\colon P \to S$ be a principal $T$-bundle with connection $\theta \in \Omega^1(P;\mathfrak{t})$.
Assume
\begin{equation}\label{eqn:curvAssumption}
d\theta = {p}\cdot \pi^*\omega_S
\end{equation}
for some fixed ${p}\in \mathfrak{t}$. Let $(V,g_V,\omega_V)$ be a toric almost K\"ahler manifold for the same torus and moment map $\mu\colon V\to \Delta\subset \mathfrak{t}^*$. Pick $c\in \mathbb{R}$ with (here $\langle,\rangle$ is evaluation)
\begin{equation}\label{eqn:cCondition}
P(v)\coloneqq \langle v,{p} \rangle + c >0\qquad \forall v\in \Delta.
\end{equation}

Given this data, the generalized Calabi construction determines a symplectic structure $\omega_M$ on the total space of the associated bundle
\begin{equation}
	M \coloneqq P\times_T V\xrightarrow{\pi} S.
\end{equation}
On the free stratum $V^0=\mu^{-1}(\Delta^0)$ over the interior of the Delzant polytope let $\alpha\colon TV^0\to \mathfrak{t}$ be the $g$-orthogonal projection onto the orbits. The linear map
\[
T_p P \times T_v V^0\to\mathfrak{t},\quad
(X,Y) \mapsto \theta(X) + \alpha(Y)
\]
is invariant under the action of the tangent group and thus induces a $1$-form $\theta^0$ on $M^0\coloneqq P\times_T V^0$.
\Marginpar{In local coordinates where $P$ is trivial, the connection $\theta=A+\theta_{\mathrm{MC}}$ is given by a $\mathfrak{t}$-valued $1$-form $A$ on the base, where $\theta_{\mathrm{MC}}\in \Omega^1(T;\mathfrak{t})$ denotes the Maurer-Cartan form. Then $\theta^0=A+dt$, where $t\colon V^0\to T$ are angle coordinates.}
The moment map factors over the projection to $\mu\colon M\to\Delta$. Set
\begin{equation}\label{connectionAssoc}
	\omega_M = P(\mu)\pi^*\omega_S+ \langle d\mu\wedge \theta^0\rangle.
\end{equation}
Generally, the $g$-orthogonal projection $\alpha$ is a map $T_v V \to \mathfrak{t}/\mathfrak{t}_v$ up to the isotropy Lie algebra $\mathfrak{t}_v$. Since $d\mu^\xi_v$ vanishes for $\xi \in \mathfrak{t}_v$ the definition of $\langle d\mu\wedge\theta^0 \rangle$ naturally extends so that $\omega_M$ is also defined over all of $M$. \eqref{eqn:curvAssumption} implies that $\omega_M$ is closed.

When $S$ has an almost K\"ahler metric $(J_S,g_S,\omega_S)$ we get  an almost K\"ahler metric on $M$ as follows. Let $\mathbf{G}$ be the metric on $\Delta^0 \subset \mathfrak{t}^*$ that turns $\mu$ into a Riemannian submersion, 
let $\mathbf{H}$ be the dual metric on the cotangent bundle of $\Delta^0$. Precomposing with $\mu$ we obtain pairings $\mathbf{G}_p\colon \mathfrak{t}^*\otimes \mathfrak{t}^*\to \mathbb{R}$, $\mathbf{H}_p\colon \mathfrak{t}\otimes \mathfrak{t}\to \mathbb{R}$ at each  $p\in M$. Then\Marginpar{This structure is integrable if those on $S$ and $V$ are integrable, see~\cite[Theorem~2]{MR2144249}.}
\begin{equation}\label{CalabiMetric}
	g_M\coloneqq P(\mu)\pi^*g_S + \mathbf{G}(d\mu\otimes d\mu) + \mathbf{H}(\theta^0\otimes\theta^0).
\end{equation}

\begin{rem}\label{ToricMetric}
The metric on $V$ is determined by $\mathbf{G}$: recall that every metric $\mathbf{G}$ on $\Delta^0$ subject to appropriate boundary conditions (see~\cite[Proposition~1]{MR2144249} or \eqref{one}, \eqref{three} below) compactifies to an $\omega_V$-compatible almost complex structure $g_V$. Recall also that, up to symplectomorphism, 
any metric on $V$ arises this way \cite[Lemma~3]{MR2144249}.
\end{rem}

\subsection{Ruled manifolds}

We now restrict to $T=S^1$.
We shall say that a Hermitian line bundle $L$ with connection has \emph{degree} $p\in \mathbb{R}$ if $R^L = p\omega_S$ for the curvature.

\begin{rem}
Modifying $\omega_S$ slightly, such line bundles always exist for closed $S$. Indeed, an arbitrary small perturbation of $\omega_S$ is a symplectic form that represents a rational cohomology class, so some $q\omega_S$ with $q\in \mathbb{Q}$ represents an integer cohomology class (see~\cite[Observation~4.3]{MR1356781}). 
The corresponding Konstant--Souriau line bundle has the required properties, with $p=1/q$.
Another important class of examples is when $S$ is a Riemann surface. Here, holomorphic line bundles are determined by their degree $p\in \mathbb{Z}$ with $c_1(L)=p[\omega_S]$. Using the $\partial\bar\partial$\nobreakdash-Lemma we find a Hermitian connection whose curvature $2$-form is precisely $p\omega_S$.
\end{rem}

Let $V=\C P^1$ with Fubini--Study symplectic form $\omega_{\mathrm{FS}}$ and Delzant polytope ${\Delta=[0,1]}$. As in \eqref{eqn:cCondition} choose $c$ with $P(x)\coloneqq px+c$  positive on $[0,1]$.

\begin{defn}
The \emph{ruled manifold} belonging to $(L\to S,c)$ is $M \coloneqq \mathbb{P}(L\oplus \C)$ equipped with the symplectic form $\omega_{M,c}$ from \eqref{connectionAssoc}. 
\end{defn}

Hence $M$ is obtained by compactifying each fiber of $L$ to a sphere. Following~\cite{MR2807093} we assume also that the base $S$ has constant Hermitian scalar curvature.

\Marginpar{$M$ is obtained by gluing $L$ to its conjugate bundle $\bar{L}$ away from their zero sections using inversion at the unit circle, defined by the Hermitian metric.}

\subsection{Existence problem}\label{ssec:ExistenceProblem}

Since the scalar curvature is $S^1$\nobreakdash-invariant, it is our strategy to consider only conformal variations $u=\varphi\circ \mu$ for $\varphi\colon [0,1]\rightarrow \mathbb{R}^+$.

\begin{thm}\label{thm:babycase}
Let $({M^m=\mathbb{P}(L\oplus\C)}, \omega_{M,c})$ be a ruled manifold over a closed K\"ahler manifold ${(S^{m-2},g_S,\omega_S)}$ of constant positive scalar curvature.
Choose ${u=a\mu+b}$ with $a,b>0$. Rescaling the volume of $S$ if necessary, there exists for $c$ sufficiently large \Marginpar{$c$ only depending on $a,b$ and the degree $p$} a compatible $S^1$\nobreakdash-invariant K\"ahler metric $g$ on $M$ so that $\tilde{g}\coloneqq u^{-2}g$ has constant Hermitian scalar curvature \textup(for $\tilde\omega=u^{-2}\omega_{M,c}$\textup).
\end{thm}


This solves the Existence Problem~\ref{EXproblem} on ruled manifolds.
They are examples of almost Hermitian manifolds of positive fundamental constant (see~Proposition~\ref{value_of_C}) which are not covered by the results of the previous section. Instead of rescaling the base one may also change the Fubini--Study form on the fibers by a fixed factor.

\begin{rem}
By Apostolov--Calderbank--Gauduchon--T{\o}nnesen-Friedman, ruled manifolds with $c$ sufficiently large also admit an extremal metric \cite[Theorem~4]{MR2807093}.
\end{rem}

As recalled above in Remark~\ref{ToricMetric}, compatible $S^1$\nobreakdash-invariant metrics on $(\C P^1,\omega_{\mathrm{FS}})$ correspond to smooth functions $\mathbf{H}\colon [0,1] \rightarrow \mathbb{R}$ satisfying the boundary conditions
\begin{align}
\begin{aligned}\label{one}
\mathbf{H}(0)&=\mathbf{H}(1)=0,\\
\mathbf{H}'(0)&=2=-\mathbf{H}'(1),
\end{aligned}\\
\mathbf{H}(x)>0\quad(0<x<1).\label{three}
\end{align}
Then $\mathbf{H}$ and $g_S$ determine an $\omega_M$-compatible almost K\"ahler metric $g_M$ via \eqref{CalabiMetric}. The metric $g$ we seek will be of the form \eqref{CalabiMetric} and is hence determined by a function $\mathbf{G}^{-1}\coloneqq\mathbf{H}$ satisfying \eqref{one}, \eqref{three}. Let $t\colon V^0 \to S^1$ be the angle coordinate on the round sphere $\C P^1 \setminus \{N,S\}$. Let $(x^i)$ be local coordinates on $S$ over which $L$ is trivialized. The connection then corresponds to a local $1$-form $A=A_i dx^i$ on the base. We have local coordinates $(x^i,\mu,t)$  on $M^0$ in which the induced $1$-form can be written $\theta^0 = A + dt$. From \eqref{connectionAssoc} we then get the volume form
\begin{equation}\label{volumeCalabi}
\vol_M = \frac{\omega_{M,c}^n}{n!} = \frac{1}{n}P(\mu)^{n-1} \vol_S\wedge d\mu\wedge dt.
\end{equation}
According to \eqref{CalabiMetric}, the local expression for the metric $g$ is:
\[
\big[P(\mu)g_{ij}^S + \mathbf{H}(\mu)A_i(\mu)A_j(\mu) \big] dx^i dx^j + \mathbf{G}(\mu) d\mu d\mu + \mathbf{H}(\mu) dtdt + 2\mathbf{H}(\mu) A_i(x) dx^idt
\]
From this we see $d\mu^{\sharp_g}=\mathbf{H}(\mu)\frac{\partial}{\partial\mu}$. Putting this into the formula $d(i_{\grad_g f} \vol_M) = -\Delta^g(f)\vol_M$ we get for the Hodge--de~Rham Laplacian of the moment map
\begin{equation}\label{LaplacianMoment}
\Delta^g(\mu) = -\frac{(P^{n-1} \mathbf{H})'(\mu)}{P(\mu)^{n-1}}. 
\end{equation}
Note also the general formula $\Delta^g(\varphi\circ \mu) = \varphi'(\mu)\Delta^g \mu - \varphi''(\mu)|d\mu|^2_g$.

We shall refer the analytic part of the proof to the next subsection.

\begin{proof}[Proof of Theorem~\ref{thm:babycase}]
According to Proposition~\ref{ODEpart} below with $\lambda\coloneqq b/a$ we find unique $A,B > 0$ and $f\in C^\infty([\lambda,1+\lambda])$ strictly positive on the interior satisfying \eqref{ODE1}, \eqref{ODE2}. 
Rescaling the volume, we assume that $B$ is the scalar curvature of $S$.

By \cite[Lemma~9]{MR2807093} the Hermitian scalar curvature of $(M,g,\omega_{M,c})$ is (omitting the argument $\mu$ and where the derivatives are taken as a functions of $x \in [0,1]$)
\begin{equation}\label{ScalCalabi}
	s^H = \frac{B}{P} - \frac{(P^{n-1} \mathbf{H})''}{P^{n-1}}.
\end{equation}
Combining \eqref{Eqn:HermScalPoly}, \eqref{LaplacianMoment}, and \eqref{ScalCalabi} we get for $\tilde{g}=u^{-2}g$, where $u=\varphi\circ\mu$:
\[
	\tilde{s}^H = \varphi^2\frac{B}{P} - \varphi^2\frac{(P^{n-1}\mathbf{H})''}{P^{n-1}} + m\varphi\varphi'\frac{(P^{n-1}\mathbf{H})'}{P^{n-1}} + m\left(\varphi\varphi''-(\varphi')^2\right)\mathbf{H}
\]
If $\varphi(x)=ax+b$ and defining $f(x+\lambda)= P(x)^{n-1}\mathbf{H}(x)$ we see from \eqref{ODE1} that we have found a solution $\mathbf{H}$ to this equation where $\tilde{s}^H = a^2A$. Condition \eqref{one} is \eqref{ODE2} and \eqref{three} is just the positivity of $f$.
\end{proof}

\begin{prop}\label{value_of_C}
For $c$ sufficiently large, the metric $\tilde{g}$ on $M=P(L\oplus \C)$ in Theorem~\textup{\ref{thm:babycase}} has positive constant Hermitian scalar curvature. When $n=\dim_\C M=2$
\begin{equation}\label{GauduchonDegreeCalabi}
	C(M,J,[g])=\frac{2s^H_S+8c+4p}{2c+p}.
\end{equation}
\end{prop}
\begin{proof}
The coarea formula applied to the submersion $\mu\colon M\to [0,1]$
and volume form $\omega_M^n/n!$ shows that for measurable $f\colon [0,1]\to \mathbb{R}$ 
we have (see also~\cite[p.~17]{MR2807093})
\Marginpar{In general, when $f$ is any measurable function \emph{on $M$} we have $\int_M f \frac{\omega_M^n}{n!} = \int_\Delta \int_{\mu^{-1}(x)} fP(x)\frac{(\pi^*\omega_S)^{n-1}}{(n-1)!}dx$}
\begin{equation}\label{coarea}
	\int_M (f\mu) \frac{\omega_M^n}{n!} = \int_0^1 f(x) \operatorname{vol}(\mu^{-1}x) dx.
\end{equation}
On $\mu^{-1}(x)$ the form \eqref{connectionAssoc} restricts to $P(x)\pi^*\omega_S$, hence $\operatorname{vol}(\mu^{-1}x)=P(x)^{n-1}\operatorname{vol}(B)$. Without loss we may suppose $\operatorname{vol}(B)=1$. The K\"ahler metric $(\omega_M,g)$ is Gauduchon and normalizing \eqref{GauduchonDegree} to unit volume gives
\[
	C(M,J,[g]) = \frac{1}{\operatorname{vol}(M)}\int_M s^H \frac{\omega_M^n}{n!}
\]
Recall $P(x)=px+c$. Evaluating using \eqref{coarea} gives
\begin{align*}
\operatorname{vol}(M) &= \int_M \frac{\omega_M^n}{n!} = \operatorname{vol}(B)\int_0^1 P(x)^{n-1} dx = \frac{(p+c)^n - c^n}{pn}\\
\int_M s^H\frac{\omega_M^n}{n!} & \overset{\eqref{ScalCalabi}}{=} \int_0^1 \frac{s^H_S - (P\varphi)''(x)}{P(x)} \operatorname{vol}(\mu^{-1}x) dx\\
&= \operatorname{vol}(B)\int_0^1 \left[s^H_S - (P\varphi)''(x) \right] P(x)^{n-2} dx
\end{align*}
Integrating by parts and inserting the boundary condition \eqref{one} gives
\begin{align*}
&\int_M s^H\frac{\omega_M^n}{n!} =
s^H_S \frac{(p+c)^{n-1} - c^{n-1}}{p(n-1)} - \left[(P \varphi)'(x)P(x)^{n-2}\right]^1_0 + \int_0^1(P\varphi)'(x)P(x)^{n-2}dx\\
&= s^H_S \frac{(p+c)^{n-1} - c^{n-1}}{p(n-1)} + 2c^{n-1}+2(p+c)^{n-1} + p\int_0^1 \varphi'(x)P(x)^{n-2}dx + \int_0^1 \varphi(x) P(x)^{n-1}dx
\end{align*}
The last integral is $\geq 0$ since $\varphi$ and $P$ are non-negative on $[0,1]$ by \eqref{three}, while the second to last integral is of order $O(c^{n-2})$. The result follows. If $n=2$ then $\int_0^1(P\varphi)'(x)P(x)^{n-2}dx=0$ by the boundary conditions, and we obtain \eqref{GauduchonDegreeCalabi}.
%
%
\end{proof}

\subsection{Analytic part}\label{ssec:analyticpart}

We will use Landau notation in a narrower sense than usual.
For us $O(c^n)$ stands for an arbitrary Laurent polynomial in $c$
of degree less than $n$ with smooth coefficients (constants unless we are working with functions of $x$). 
The reason for this restriction is that we need to ensure the rules
\begin{align}
\int O(c^n)dx &= O(c^n),
&\frac{d}{dx} O(c^n) &= O(c^n), 
& x^k O(c^n) &= O(c^n).\label{generalLandau}
\end{align}

Elementary properties of the determinant (Cramer's rule) show:

\begin{lem}\label{asympLemma}
Let $M_c=(M^1_c , \cdots , M^N_c)\in \mathbb{R}^{N\times N}$ be a matrix whose columns are functions of $c$ satisfying $M^j_c = M^j_\infty c^{n_j} + O(c^{n_j-1})$.
Then for $M_\infty=(M^1_\infty, \cdots, M^N_\infty)$
\[
	\det(M_c) = \det(M_\infty) c^{n_1+\cdots+n_N} + O(c^{n_1+\cdots+n_N-1}).
\]
In particular, when $M_\infty$ is invertible it follows that $M_c$ is invertible for $c$ sufficiently large. Assume $b_c = b_\infty c^m + O(c^{m-1}) \in \mathbb{R}^N$ and let $x_c$ and $x_\infty$ be the respective solutions to $M_cx_c=b_c$ and $M_\infty x_\infty=b_\infty$. Then for the $j$\nobreakdash-th entry
\[
x^j_c=x^j_\infty c^{m-n_j} + O(c^{m-n_j-1}).
\]
\end{lem}

To complete the proof of Theorem~\ref{thm:babycase} it remains to prove the following:

\begin{prop}\label{ODEpart}
For any $m=2n \in \mathbb{N}$, $\lambda>0$, and $P(x)=px+c$ with $c$ sufficiently large, there exists a unique solution $(A,B,f) \in \mathbb{R}_{>0}\times \mathbb{R}_{>0} \times C^\infty([\lambda,1+\lambda])$ of
\begin{equation}\label{ODE1}
x^2f''(x) -mxf'(x) + mf(x) = -A\cdot P(x-\lambda)^{n-1} + B\cdot x^2 \cdot P(x-\lambda)^{n-2}
\end{equation}
with initial values
\begin{align}\label{ODE2}
f(\lambda)&=f(1+\lambda)=0
	& f'(\lambda)&=2c^{n-1}
	& f'(1+\lambda)&=-2(p+c)^{n-1}.
\end{align}
Moreover, $f$ is strictly positive on $]\lambda,1+\lambda[$.
\end{prop}

Care must be taken to prove the positivity since near $\lambda$ the solution functions $f_c$ for $c\rightarrow \infty$ could oscillate around zero into the negative, even if the limiting function $f_\infty$ is strictly positive on the interior.

\begin{proof}
The corresponding homogeneous equation has solutions $x, x^m$. Applying `reduction of order' \cite[p.~242]{tenenbaum1985ordinary} we obtain all solutions of \eqref{ODE1} in the form
\begin{equation}\label{EQf}
	f(x)= (u_1(x)+C)x+(u_2(x)+D)x^m
\end{equation}
for arbitrary constants $C, D$ and where 
\begin{align}\label{uIntegrals}
u_1(x) &= \frac{-1}{m-1}\int Q(x) x^{-2}dx
&u_2(x) &= \frac{1}{m-1}\int Q(x)x^{-m-1} dx
\end{align}
are chosen primitives. Since the integrands are Laurent polynomials in $x$, the singularities are at zero and we conclude that all solutions $f(x)$ are smooth on $[\lambda,1+\lambda]$.

For each parameter $c$ we consider the initial value problem \eqref{ODE2}. From \eqref{EQf} we see that these are \emph{linear} equations for $\vec{x}=(A,B,C,D)$. Let $M_c$ denote the corresponding $4\times 4$ matrix, so that we are looking for solutions of
\begin{equation}\label{Mc}
M_c\cdot \vec{x}_c
=
(0,0,2c^{n-1},-2(p+c)^{n-1})^\mathrm{T}.
\end{equation}
Once we show that $M_c$ is invertible, this singles out a unique solution $f_c$. By Lemma~\ref{asympLemma} to show that $M_c$ is invertible and to understand the asymptotics of the solution $\vec{x}_c$ we need only keep track of the dominant powers of $c$ in front of the variables. 
Putting 
\[
	Q(x) = Bc^{n-2}x^2 - Ac^{n-1} + A\cdot O(c^{n-2}) +O(c^{n-3})
\]
into \eqref{uIntegrals} and remembering \eqref{generalLandau} gives
\begin{align*}
u_1(x) &= \frac{-1}{m-1}\left( Ac^{n-1}x^{-1} + Bc^{n-2}x \right) + A\cdot O(c^{n-2})+O(c^{n-3})\\
u_2(x) &= \frac{1}{m-1} \left( \frac{Bc^{n-2}}{2-m}x^{2-m} + \frac{Ac^{n-1}}{m} x^{-m} \right)+ A\cdot O(c^{n-2})+O(c^{n-3})
\end{align*}
and so
\begin{equation}\label{fAsy}
	f(x) = A\left( \frac{-c^{n-1}}{m}+O(c^{n-2}) \right)
	+B\left( \frac{-c^{n-2}}{m-2} + O(c^{n-3}) \right)x^2
	+Cx
	+Dx^m.
\end{equation}
Inserting this into \eqref{ODE2} leads to the matrix 
\[
M_c= \begin{pmatrix}
\frac{-c^{n-1}}{m} + O(c^{n-2}) & \frac{-c^{n-2}}{m-2}\lambda^2+O(c^{n-3}) & \lambda & \lambda^m\\
\frac{-c^{n-1}}{m} + O(c^{n-2}) & \frac{-c^{n-2}}{m-2}(1+\lambda)^2+O(c^{n-3}) & 1+\lambda & (1+\lambda)^m\\
O(c^{n-2}) & \frac{-c^{n-2}}{n-1}\lambda+O(c^{n-3}) & 1 & m\lambda^{m-1}\\
O(c^{n-2}) & \frac{-c^{n-2}}{n-1}(1+\lambda)+O(c^{n-3}) & 1 & m(1+\lambda)^{m-1}
\end{pmatrix}.
\]
Also $b_c = (0,0,2,-2) c^{n-1} + O(c^{n-2})$. The matrix
\[
M_\infty = \begin{pmatrix}
\frac{-1}{m} & \frac{-\lambda^2}{m-2} & \lambda & \lambda^m\\
\frac{-1}{m} & \frac{-(1+\lambda)^2}{m-2} & 1+\lambda & (1+\lambda)^m\\
0 & \frac{-\lambda}{n-1} & 1 & m\lambda^{m-1}\\
0 & \frac{-(1+\lambda)}{n-1} & 1 & m(1+\lambda)^{m-1}
\end{pmatrix}
\]
is invertible. The equation $M_\infty\cdot x_\infty = (0,0,2,-2)^\mathrm{T}$ has the solution
\[
x_\infty = \big(2m\lambda(1+\lambda), 2(m-2),2(1+2\lambda),0\big)
\]
 which by Lemma~\ref{asympLemma} gives us
\begin{align*}
	A_c &= 2m\lambda(1+\lambda)+O(c^{-1}),
	&B_c &= 2(m-2)c + O(c^0),\\
	 C_c &= 2(1+2\lambda)c^{n-1} + O(c^{n-2}),
	& D_c &= O(c^{n-2}).
\end{align*}
Putting this into \eqref{fAsy} shows
\[
	f_c(x+\lambda) = 2c^{n-1}x(1-x)  + O(c^{n-2}).
\]
Hence
\[
	\frac{f'_c(x+\lambda)}{2c^{n-1}} = 1-2x + O(c^{-1})
\]
uniformly in $x$. 
It follows that for $c$ sufficiently large $f_c$ has precisely one extreme point on $[\lambda,1+\lambda]$ (close to $1/2+\lambda$). From \eqref{ODE2} we see that this must be a maximum and hence $f_c$ is positive on $[\lambda,1+\lambda]$.
\end{proof}


\begin{figure}[h]
\centering
\begin{tikzpicture}
 \begin{axis}[
  axis x line=center,
  axis y line=center,
  xtick={0,0.5,1},
  ytick={1,2},
  xlabel={$x$},
  ylabel={},
  xlabel style={below right},
  ylabel style={above left},
  xmin=-0.1,
  xmax=1.1,
  ymin=-0.5,
  ymax=3]
    \addplot[domain=0:1] {1.60801 +  x* (4.57629 + x* (-5.19246 + (-2.54247 - 1.27123 *x)* x)) + (2.31986 +     4.63973*x)*ln(0.5+x)} node[above,pos=0.5] {$f_c(x+\lambda)$};
    \addplot[domain=0:1,dashed]{6*x*(1-x)} node[above,pos=0.53] {$2c^{n-1}x(1-x)$};
  \end{axis}
\end{tikzpicture}
\caption{Plot of the solution $f_c$ for $p=3,\lambda=\frac12,m=4,c=3$ and of the `ideal solution' at infinity}
\end{figure}
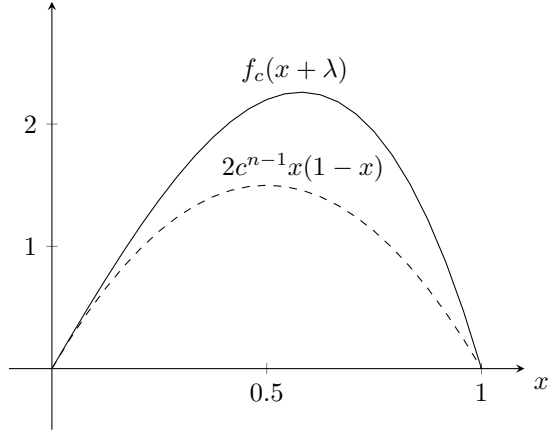

\section{Moment Map Setup}\label{sec:MomentMap}

In this section we give a moment map interpretation of the Existence Problem~\ref{EXproblem} inspired by the work of Apostolov and Maschler~\cite{Apostolov:2015aa}.
This leads to the familiar existence and uniqueness conjectures formulated in terms of geometric invariant theory, as well as to a version of the Futaki invariant.  Our method applies to closed symplectic manifolds $(M^m,\omega)$ admitting a symmetry given by a Hamiltonian vector field, meaning we assume $\mathcal{AC}^f(\omega)\neq \emptyset$ below. An example is a manifold with Hamiltonian circle action, as above.

First we recall the general definition ($\mathfrak{g}^*$ gets the coadjoint action):

\begin{defn}\label{def:momentmap}
A symplectic action of a Lie group on a symplectic manifold $(S,\omega)$ is \emph{Hamiltonian} if there exists a $G$-equivariant \emph{moment map} $\mu\colon X\to \mathfrak{g}^*$ with
\begin{equation}\label{eqn:momentmap}
	d\mu^\xi(X) = \omega(X,\xi^*)\qquad \forall \xi \in \mathfrak{g}, X\in TS.
\end{equation}
(write $\mu^\xi = \mu(-)(\xi) \in C^\infty(S)$ and let $\xi^* \in \mathfrak{X}(S)$ be the infinitesimal action.)
\end{defn}

\subsection{The action} 

For fixed $f\in C^\infty(M)$ with $\smallint_M f\vol = 0$ let
\begin{align}\label{notation}
u &\coloneqq e^{-nf},
&K&\coloneqq \grad_\omega u.
\end{align}

\begin{defn}
$\mathcal{AC}^f(\omega)$ is the Fr\'{e}chet manifold of $\omega$-compatible almost complex structures $J$ satisfying $\mathfrak{L}_K J = 0$.
\end{defn}




We equip $\mathcal{AC}^f(\omega)$ with the symplectic form
\begin{align}\label{FatSymplectic}
\mathbf{\Omega}_J(A,B) &= \frac12 \int_M \tr(J\circ A \circ B)e^{nf}\vol,\qquad A,B \in T_J \mathcal{AC}^f(\omega).
\end{align}


\begin{prop}\label{ACcontractible}
When non-empty, $\mathcal{AC}^f(\omega)$ is contractible. This is the case precisely when $K$ is the a Killing vector field for some metric $g$ on $M$.
\end{prop}

\begin{proof}
The usual map restricts to a retraction
\[
	\mathfrak{Met}^f(M) \to \mathcal{AC}^f(\omega),\quad
	g\mapsto J_g,\enskip\text{where}\enskip
	J_g\coloneqq A_g|A_g|^{-1},\enskip g(A_g-,-)\coloneqq\omega
\]
on the convex space of metrics $g$ satisfying $\mathfrak{L}_K g = 0$.
\end{proof}

\begin{rem}
As mentioned above, we make the \emph{assumption} that $\mathcal{AC}^f(\omega)$ is non-empty. According to Bochner, there are no non-trivial Killing fields in the case of strictly negative Ricci curvature. We have $s^H = 2|N|^2 + s^g$ and so the results of this section mainly concern the case of positive fundamental constant.
\end{rem}

\begin{defn}\label{def:Ham}
$\Ham^f(\omega) \subset \mathrm{Symp}(\omega)$ is the subgroup of Hamiltonian symplectomorphisms $\phi$ satisfying $f\circ \phi = f$ (equivalently $\phi_*$ preserves $\grad_\omega f$). 
\end{defn}

We recall that by definition the Lie algebra of the Hamiltonian symplectomorphisms are the Hamiltonian vector fields $d\varphi=-\iota_X\omega$, where $\varphi\in C^\infty(M)$.
The Lie algebra $\LieHam$ consists of Hamiltonian vector fields with ${[\grad_\omega f,X]=0}$. Let $C^\infty_f(M)$ be the space of $\varphi\in C^\infty(M)$ with constant Poisson bracket $\{f,\varphi\}$. Then $\LieHam$ is canonically identified with $C^\infty_f(M)/\mathbb{R}$.
The adjoint action of $\phi\in \Ham^f(\omega)$ on $\varphi \in C_{f}^\infty(M)/\mathbb{R}$ can be written $(\phi^{-1})^*\varphi$. On $C^\infty_{f}(M)$ consider
\begin{equation}\label{inner_product}
\langle h_1,h_2 \rangle_{e^{(2+n)f}} = \int_M h_1h_2e^{(2+n)f}\vol.
\end{equation}
It places $C^\infty_f(M)/\mathbb{R}$ in duality with ${C^\infty_{0,f}(M)\coloneqq\{\varphi\in C^\infty_f(M)\mid \langle \varphi,1\rangle_{e^{(2+n)f}} = 0\}}$. 
We have an isomorphism to $C^\infty_{0,f}(M)\to C^\infty_f(M)/\mathbb{R}$ with inverse
\[
C^\infty_f(M)/\mathbb{R}\to C^\infty_{0,f}(M),\quad
	\varphi\mapsto \mathring{\varphi} \coloneqq \varphi - \frac{\langle \varphi,1\rangle_{e^{(2+n)f}}}{\langle1,1\rangle_{e^{(2+n)f}}}
	=
	\varphi - \frac{\smallint_M \varphi e^{(2+n)f}\vol}{\smallint_M e^{(2+n)f}\vol}.
\]
For $\varphi,\psi \in C^\infty_f(M)$ note the formula
\begin{equation}\label{dotFlip}
\langle \mathring{\varphi},\psi \rangle_{e^{(2+n)f}}=\langle \varphi, \mathring{\psi}\rangle_{e^{(2+n)f}}.
\end{equation}

Since $e^{nf}\vol$ is preserved by $\phi$, the action of $\Ham^f(\omega)$ on $\mathcal{AC}^f(\omega)$ by ${\phi_*\circ J\circ \phi_*^{-1}}$ preserves the symplectic form \eqref{FatSymplectic}. We will show that it is Hamiltonian.

\subsection{Technical preparations}

\begin{lem}[see~{\cite[Lemma~1.3]{lejmi2010metriques}}]\label{lem:Killing}
Let $(J,g,\omega)$ be almost K\"ahler.
Suppose the symplectic gradient $K=\grad_\omega u$ is a $g$-Killing field.
Then the $J$-anti-invariant part $(D^g Jdu^\sharp)^{J,-}$ is anti-symmetric.
\end{lem}
\begin{proof}
Because $\mathfrak{L}_K \omega = 0$ is automatic, the field $K=Jdu^\sharp$ is Killing precisely when it is holomorphic. Therefore, combined with the fact that $D^g$ is torsion-free,
\[
	0=\mathfrak{L}_K J = D_{K}^gJ - [D^gK,J].
\]
So $(D^g K)^{J,-} = \frac12J[D^gK,J] = \frac12 JD^g_KJ = -\frac{1}{2} D^g_{JK}J$ which is anti-symmetric.
\end{proof}


Consider a path $J_t \in \mathcal{AC}(\omega)$ representing $\dot{J} = \left.\frac{d}{dt}\right|_0 J_t$. Write $g_t = \omega(\cdot,J_t\cdot)$. The variation of the scalar curvature is given by the Mohsen formula:

\begin{prop}[see~\cite{mohsen}]\label{mohsenprop}
$\left.\frac{d}{dt}\right|_0 s^H_{g_t} = - \delta J (\delta \dot{J})^\flat$.
\end{prop}

In this formula the codifferential of an endomorphism $A$ is defined by
\begin{align}
	g(\delta A, X)=\delta \langle A, X\rangle + g(A,D^g X),
	\quad X\in \mathfrak{X}(M)
	\label{codifferentialEndo}
\end{align}
using the evaluation pairing $\langle,\rangle$.
For $1$-forms $\alpha,\beta$ we note also the simple formulas
\begin{align}
g(\alpha^\sharp, A(X)) &= g(A^*, \alpha\otimes X),\label{gAdjoint}\\
\left.\frac{d}{dt}\right|_0g_t(\alpha,\beta) &= -g(\alpha,\dot{J}J\beta),\label{gLimit}
\end{align}
which are used in the proof of our main technical lemma:

\begin{lem}
For the metrics $\tilde{g}_t = e^{2f}\omega(\cdot,J_t\cdot)$ and any $h \in C^\infty(M)$ we have
\begin{align}
\left.\frac{d}{dt}\right|_0 \int_M s^H_{\tilde{g}_t} h e^{(2+n)f} \vol
= \int_M g(\dot{J}, D^g Jdh^\sharp) e^{nf}\vol.\label{Moment1}
\end{align}
\end{lem}
\begin{proof}
Recall $u:=e^{-nf}$. By \eqref{Eqn:HermScal} the scalar curvature of the conformal variation is
\[
	s^H_{\tilde{g}_t} = e^{-2f}\left(s^H_{g_t} + m\Delta^{g_t}(f)\right)
	= u^{2/n} s^H_{g_t}  - 2u^{2/n-1}\Delta^{g_t}(u) - 2u^{2/n-2}|du|^2_{g_t}.
\]
Putting this and $e^{(2+n)f} = u^{-1-2/n}$ into the left hand side of \eqref{Moment1} gives
\begin{align*}
\left.\frac{d}{dt}\right|_0 \int_M s^H_{g_t} h u^{-1} \vol 
- 2 \int_M \Delta^{g_t}(u) hu^{-2}\vol
- 2\int_M g_t(du,du) hu^{-3}\vol.
\end{align*}
Applying Proposition~\ref{mohsenprop} and \eqref{gLimit} to the second and third summand we get
\begin{align*}
-\int_M \delta J(\delta \dot{J})^\flat hu^{-1}\vol + 2\int_M g(du, \dot{J}Jd(hu^{-2})) \vol
+ 2\int_M g\left(du, \dot{J}Jdu \right)hu^{-3} \vol
\end{align*}
which, in view of \eqref{codifferentialEndo} and \eqref{gAdjoint} becomes
\begin{align}\label{tempMoment}
\int g\left(\dot{J}, D^g J d(hu^{-1})^\sharp + 2 du\otimes Jd(hu^{-2})^\sharp + 2hu^{-3} du\otimes Jdu^\sharp \right) \vol.
\end{align}
Now expand using the Leibniz rule:
\begin{equation}\label{Leibniz}
\begin{aligned}
D^g Jd(hu^{-1})^\sharp=&2u^{-3}h du\otimes Jdu^\sharp - u^{-2} du\otimes Jdh^\sharp - u^{-2}dh\otimes Jdu^\sharp\\& + u^{-1}D^g(Jdh^\sharp) - u^{-2}h D^gJdu^\sharp\\
du\otimes Jd(hu^{-2})^\sharp =& u^{-2}du\otimes Jdh^\sharp - 2u^{-3} h du\otimes Jdu^\sharp
\end{aligned}
\end{equation}
From \eqref{gAdjoint} we see $g(\dot{J}, du\otimes Jdh^\sharp) = g(\dot{J}, dh\otimes Jdu^\sharp)$. Moreover, Lemma~\ref{lem:Killing} implies $g(\dot{J}, D^g Jdu^\sharp)=0$ since $\dot{J}$ is symmetric and $J$-anti-invariant, while the $J$-anti-invariant part of $D^g Jdu^\sharp$ is anti-symmetric. Inserting \eqref{Leibniz} into \eqref{tempMoment} and applying these facts then gives the right hand side of \eqref{Moment1}.
\end{proof}

\subsection{Proof of main theorem}
We write $g_{f,J}\coloneqq e^{2f}\omega(\cdot,J\cdot)$ and $g_J \coloneqq g_{0,J}$.

\begin{thm}\label{thm:momentmap}
Let $(M,\omega)$ be a closed symplectic manifold with Hamiltonian vector field ${K=\grad_\omega u}$, ${u=e^{-nf}}$, ${\smallint_M f\vol = 0}$, and $\mathcal{AC}^f(\omega) \neq \emptyset$.
The action of $\mathrm{Ham}^f(\omega)$ on $\mathcal{AC}^f(\omega)$ with symplectic form \eqref{FatSymplectic} is Hamiltonian with moment map
\begin{align}\label{OurMomentMap}
\mu\colon \mathcal{AC}^f(\omega) \times C^\infty_{0,f}(M) \to \mathbb{R},\quad
\mu^{\mathring{h}}(J) = \int_M s_{g_{f,J}}^H \mathring{h} e^{(2+n)f}\vol.
%
\end{align}
Here the Hermitian scalar curvature of $g_{f,J}$
is viewed as a functional using \eqref{inner_product}.
\end{thm}

Identifying $\LieHam=C^\infty_f(M)/\mathbb{R}$ 
and using \eqref{dotFlip} we may rewrite
\begin{equation}\label{momentsecond}
	\mu\colon \mathcal{AC}^f(\omega) \to (C^\infty_f(M)/\mathbb{R})^*,\quad
	\mu(J)=\int_M \mathring{s}_{g_{f,J}}^H h e^{(2+n)f}\vol.
\end{equation}

\begin{proof}
We must check that for any tangent vector $\dot{J} \in T_J \mathcal{AC}^f$ and $h\in C^\infty_{0,f}(M)$
\[
	\mathbf{\Omega}_J(h^*_J,\dot{J}) = d\mu^h(\dot{J}).
\]
Here $h^*_J= -\mathfrak{L}_Z J  \in T_J \mathcal{AC}^f$ for $Z=\grad_\omega h$ denotes the infinitesimal action of $h$ at the point $J$. 
In terms of the adjoint of $D^g_\cdot Z \in \End(TM)$ with respect to $g=g_J$, the infinitesimal action can be rewritten as $h^*J = -J\circ (D^g Z + (D^g Z)^*)$ and so 
\begin{align*}
	\mathbf{\Omega}_J(h^*_J, \dot{J})
	&= \frac12\int_M \left(\tr(D^g Z \circ \dot{J}) + \tr((D^g Z)^*\circ \dot{J})\right) e^{nf}\vol\\
	&= \int_M \tr(D^gZ \circ \dot{J}) e^{nf}\vol.
\end{align*}
This is the right hand side of \eqref{Moment1}, as $Z=Jdh^\sharp$, while the left hand side of \eqref{Moment1} is simply $d\mu^h(\dot{J})$. From $\phi^*g_{\phi\cdot J, f} = g_{J,\phi^*f}$ we get $\phi^*s_{\phi\cdot J, f}^H = s^H_{J,\phi^*f}$. Now $f\circ \phi = f$ by Definition~\ref{def:Ham} and so $\phi^*\mu(\phi\cdot J) = \mu(J)$, proving that \eqref{OurMomentMap} is also equivariant.
\end{proof}

The zeros of the moment map $\mu$ are $J\in \mathcal{AC}^f(\omega)$ such that the metric $g_{f,J}$ is of constant Hermitian scalar curvature.
The geometric invariant theory formal picture suggests then the existence of a unique almost-K\"ahler metric in $\mathcal{AC}^f(\omega)$ conformal to a constant Hermitian scalar curvautre metric,
modulo the action of $\mathrm{Ham}^f(\omega)$, in every ``{\textit{stable}}" ``{\textit{complexified}}" orbit of the action of $\mathrm{Ham}^f(\omega)$.

\begin{rem}
In~\cite[Remark 1]{Apostolov:2015aa}, the zeros of the moment map are metrics $g_{f,J}$ with $s_{g_{f,J}}^g+|N|^2_{g_{f,J}}$ is constant,
where $s_{g_{f,J}}^g$ is the Riemannian scalar curvature of $g_{f,J}$.
\end{rem}

\begin{cor}
Minima of $\|\mu\|^2$ on $\mathcal{AC}^f(\omega)$ are conformally constant metrics.
\end{cor}

It may also be of interest to consider critical points of $\|\mu\|^2$.

\subsection{Futaki invariant}\label{ssec:Futaki}

Moment maps lead very generally to a Futaki invariant. In the context of Definition~\ref{def:momentmap}, this invariant is associated to any Lie subalgebra $\mathfrak{h}\subset \mathfrak{g}$. Letting ${S^\mathfrak{h} = \{p\in S \mid \xi^*_p = 0\enskip \forall\xi\in\mathfrak{h} \}}$ be the $\mathfrak{h}$-fixed points, the restriction ${\mu\colon {S^\mathfrak{h}} \to \mathfrak{h}^*}$ is locally constant. Assuming $S^{\mathfrak{h}}$ is connected, the common value of $\mu$ is called the Futaki invariant $\mathcal{F}^{\mathfrak{h}} \in \mathfrak{h}^*$.

Applied to our situation $\mathfrak{h}=\mathbb{R} \cdot K \subset \LieHam$ corresponding to ${u\in C^\infty_f(M)/\mathbb{R}}$. The $\mathfrak{h}$-fixed points are all of $\mathcal{AC}^f(\omega)$. Identify $\mathfrak{h}^*=\mathbb{R}$ by evaluating at $u$.

\begin{defn}
For \emph{any} $J\in\mathcal{AC}^f(\omega)$ the \emph{Futaki invariant} is given by
\begin{eqnarray*}
\mathcal{F}^f(\omega)&=&\mu^u(J)=
\langle \mathring{s}_{g_{f,J}}^H, u\rangle_{e^{(2+n)f}}
= \int_M \mathring{s}_{g_{f,J}}^H e^{2f}\vol
\\
&\overset{\eqref{Eqn:HermScal}}{=}&\int_M s_{g_J}^H\vol -\frac{\int_M s_{g_{f,J}}^H e^{(n+2)f}\vol}{\int_M e^{(n+2)f}\vol} \int_Me^{2f}\vol.   
\end{eqnarray*}
\end{defn}

The main point of the Futaki invariant, that it is independent of $J$, is a consequence of the general moment map setup and Proposition~\ref{ACcontractible}.

\begin{cor}
If $\mathcal{AC}^f(\omega)\neq \emptyset$ then there exists $J\in\mathcal{AC}^f(\omega)$ such that $g_{f,J}=e^{2f}\omega(\cdot, J\cdot)$ has constant Hermitian scalar curvature if and only if $\mathcal{F}^f(\omega)= 0$.
\end{cor}

Thus if for some $f$ the Futaki invariant $\mathcal{F}^f(\omega)$ vanishes, we have an affirmative solution to the Existence Problem~\ref{EXproblem}.

\subsection{The toric case}\label{ssec:toriccase}

Let $(M^{2n},\omega)$ be a closed symplectic manifold 
equipped with an effective Hamiltonian action of a $n$-dimensional torus $T$.
Let $z:M\to \Delta \subset \mathfrak{t}^\ast$ be the moment map, where $\Delta$ is the Delzant polytope in $\mathfrak{t}^\ast$ the dual of $\mathfrak{t}=\Lie(T).$ Denote by $\{u_1,\cdots,u_d\}$ the normals to the polytope $\Delta.$
The action of the torus $T$ is generated by a family of Hamiltonian vector fields $\{K_1,\cdots, K_n\}$ linearly independent on an open set of the $2n$-dimensional symplectic manifold $(M,\omega)$ with $\omega(K_i,K_j)=0.$
The symplectic form $\omega$ and an $\omega$-compatible $T$-invariant almost K\"ahler metric $g$ are given on $z^{-1}$ of the interior of $\Delta$ by
\begin{eqnarray*}
\omega&=&\sum_{i=1}^ndz_i\wedge dt_i,\\
g&=&\sum_{i,j=1}^nG_{ij}(z)dz_i\otimes dz_j+H_{ij}(z)dt_i\otimes dt_j+P_{ij}(z)dz_i\odot dt_j,
\end{eqnarray*}
where $G,H$ are symmetric positive definite matrix-valued functions satisfying the compatibility conditions $GH-P^2=Id$ and $HP=P^tH$ ($P^t$ is the transpose of $P$). The coordinates $z_i$ are the moment coordinates
and $t_i$ are the angle coordinates.

Denote by $H_{ij,k}=\frac{\partial H_{ij}}{\partial z_k}$ etc. It is shown~\cite{MR1988506} and \cite[(4.6)]{MR2747965} that the Hermitian scalar curvature is given by
\[
s^H=-\sum_{i,j=1}^nH_{ij,ij}.
\]
Let $u=a_1z_1+a_2z_2+\cdots+a_nz_n+a_{n+1}$ be a Hamiltonian Killing potential ($a_i$ are real numbers).
Then,
\[
Jdu=\sum_{i,l=1}^na_i P_{li}dz_i+a_iH_{il}dt_l.
\]
Hence,
\[
dJdu=\sum_{i,l=1}^na_i P_{li,j}dz_j\wedge dz_i+a_iH_{il,j}dz_j\wedge dt_l.
\]
Recall that $\Delta^g u=-g(dJdu,\omega)$. We obtain
\[
\Delta^g u=-\sum_{i,j=1}^na_iH_{ij,j},\qquad
|du|_g^2=\sum_{i,j=1}^na_ia_jH_{ij}.
\]
Hence, the conformal change equation \eqref{Eqn:HermScal} becomes
\begin{equation}\label{conformal_change_toric}
{s}^H_{g_{f,J}}=-u^{\frac{2}{n}}\sum_{i,j=1}^nH_{ij,ij}+2u^{\frac{2}{n}-1}\sum_{i,j=1}^na_iH_{ij,j}-2u^{\frac{2}{n}-2}\sum_{i,j=1}^na_ia_jH_{ij},
\end{equation}
where ${s}^H_{g_{f,J}}$ is the Hermitian scalar curvature of $g_{f,J}=e^{2f}\omega(\cdot, J\cdot)$ with $e^{-nf}=u.$
It is easy to check since $H$ is symmetric that
\[
\sum_{i,j=1}^n\left(e^{nf}H_{ij}\right)_{,ij}=\sum_{i,j=1}^ne^{nf}H_{ij,ij}-2e^{2nf}a_iH_{ij,j}+2e^{3nf}a_ia_jH_{ij}.
\]
We conclude that \eqref{conformal_change_toric} is equivalent to
\begin{equation}\label{toric_equation}
e^{(n+2)f}{s}^H_{g_{f,J}}=-\sum_{i,j=1}^n\left(e^{nf}H_{ij}\right)_{,ij}.
\end{equation}

Now, when the $g$-orthogonal distribution to the $T$-orbits is involutive (this is the case when $P=0$), $H$ has to satisfy the {\it{boundary conditions}} in~\cite[Proposition 1]{MR2144249}
and hence we can apply~\cite[Lemma 2]{Apostolov:2015aa} to get
\begin{prop}\label{apo_mas_lem}
For any $H$ satisfying the boundary conditions~\cite[Proposition 1]{MR2144249} and any affine function $\xi=\xi(z_1,\cdots,z_n),$
\begin{eqnarray*}
-\int_\Delta\left(\sum_{i,j=1}^n\left(e^{nf}H_{ij}\right)_{,ij}\right)\xi dv=2\int_{\partial\Delta}e^{nf}\xi d\mu,
\end{eqnarray*}
where $\Delta$ is the polytope and $\partial\Delta$ its boundary, $dv=dz_1\wedge\cdots\wedge dz_n$ and $d\mu$ is defined by $u_j\wedge d\mu=-dv$ for any codimension one face with inward normal $u_j$
\end{prop}

If we suppose that ${s}^H_{g_{f,J}}$ is a constant, then~\eqref{toric_equation} becomes using Proposition~\ref{apo_mas_lem} 
\begin{equation}\label{toric_equa_to_solve}
2e^{(n+2)f}\frac{ \int_{\partial\Delta} e^{nf}d\mu}{ \int_\Delta e^{(n+2)f}dv }=-\sum_{i,j=1}^n\left(e^{nf}H_{ij}\right)_{,ij}
\end{equation}

Define the Donaldson--Futaki invariant~\cite{MR1988506} for any smooth function $\xi$ to be
\[
\mathcal{F}_{\Delta,f}(\xi)=2\int_{\partial\Delta}e^{nf}\xi d\mu-2\frac{ \int_{\partial\Delta} e^{nf}d\mu}{ \int_\Delta e^{(n+2)f}dv }\int_\Delta\xi e^{(n+2)f} dv.
\]
It is straightforward from Proposition \ref{apo_mas_lem} to conclude that if there exists a solution $H$ (satisfying the boundary conditions) of~\eqref{toric_equa_to_solve},
then $\mathcal{F}_{\Delta,f}(\xi)=0,$ for any affine function $\xi=\xi(z_1,\cdots,z_n)$.
In fact, the existence of $(J,g,\omega)$ such that $g_{f,J}$ is of (positive) constant Hermitian scalar curvature can be related then to a notion of ``stability" (see for instance~\cite{Apostolov:2015aa,MR1988506})
\begin{rem}
Proposition~\ref{apo_mas_lem} implies that for any toric almost K\"ahler manifold $(M,J,g,\omega)$ with $P=0$,  
$$C(M,J,[g])=2\frac{ \int_{\partial\Delta} e^{nf}d\mu}{ \int_\Delta e^{(n+2)f}dv }>0.$$
\end{rem}


\bibliographystyle{abbrv}

\bibliography{biblio}

\end{document}